\documentclass[11pt]{amsart}
\usepackage{amssymb,amsbsy,upref,epsf,MnSymbol}
\usepackage{amsmath,amssymb,amsthm}
\usepackage{graphicx,mathtools,mathrsfs,wrapfig}
\usepackage[margin=1in]{geometry}
\usepackage{fancyhdr}
\usepackage{enumerate}

\newtheorem{theorem}{Theorem}[section]
\newtheorem{proposition}[theorem]{Proposition}
\newtheorem{lemma}[theorem]{Lemma}
\theoremstyle{remark}
\newtheorem*{remark}{Remark}
\newcommand{\R}{\mathbb{R}}
\newcommand{\N}{\mathbb{N}}

\newcommand{\Q}{\mathbb{Q}}

\renewcommand{\L}{\mathcal{L}}
\newcommand{\eps}{\varepsilon}

\newcommand{\chevron}[1]{\langle #1 \rangle}
\newcommand{\norm}[1]{\left\lVert#1\right\rVert}
\newcommand{\paren}[1]{\left( #1 \right)}
\newcommand{\bracket}[1]{\left[ #1 \right]}
\newcommand{\abs}[1]{\left\lvert #1 \right\rvert}

\newcommand{\del}{\partial}

\newcommand{\grad}{\nabla}
\newcommand{\ddt}{\frac{d}{dt}}

\newcommand{\Laplace}{\Delta}
\newcommand{\kinet}{\bracket{\del_t + v\cdot \grad_x}}
\newcommand{\bessel}{\paren{1-\Laplace_v}}
\newcommand{\loc}{\text{loc}}

\newcommand{\rest}{{\upharpoonright}}
\newcommand{\weakly}{\rightharpoonup}

\newcommand{\n}{^{-1}}
\newcommand{\indic}[1]{\chi_{\{#1\}}}

\newcommand{\Qext}{{Q_\textrm{ext}}}
\newcommand{\Qint}{{Q_\textrm{int}}}
\newcommand{\Qearly}{{Q_\textrm{early}}}
\newcommand{\Qlate}{{Q_\textrm{late}}}
\newcommand{\Rall}{\R\times\R^n\times\R^n}
\newcommand{\Ctest}{C_c^\infty}

\newcommand{\sourceexp}{{r}}
\newcommand{\txexp}{{p_1}}
\newcommand{\vexp}{{p_2}}

\newcounter{step_count}[section]
\newcommand{\step}[1]{\stepcounter{step_count} \smallskip \noindent{\textbf{Step \arabic{step_count}:} #1}}

\title[Continuity for Nonlocal Kinetic Equations]{H\"{o}lder Continuity for a Family of Nonlocal Hypoelliptic Kinetic Equations}

\author[Stokols]{Logan F. Stokols} 
\address[L. F. Stokols]{\newline Department of Mathematics, \newline The University of Texas at Austin, Austin, TX 78712, USA}
\email{lstokols@math.utexas.edu}

\date{\today}


\thanks{\textbf{Acknowledgment.} This work was partially supported by the NSF Grant DMS 1614918. }

\begin{document}

\begin{abstract}
In this work, Holder continuity is obtained for solutions to the nonlocal kinetic Fokker-Planck Equation, and to a family of related equations with general integro-differential operators.  These equations can be seen as a generalization of the Fokker-Planck Equation, or as a linearization of non-cutoff Boltzmann.  Difficulties arise because our equations are hypoelliptic, so we utilize the theory of averaging lemmas.  Regularity is obtained using De Giorgi's method, so it does not depend on the regularity of initial conditions or coefficients.  This work assumes stronger constraints on the nonlocal operator than in the work of Imbert and Silvestre \cite{is}, but allows unbounded source terms.  
\end{abstract}

\maketitle \centerline{\date}

\section{Introduction}\label{sec:intro}

We study in this paper the family of nonlocal kinetic equations
\begin{gather} \label{eq:main}
\kinet f = \L f + a, \\
\L f := \int K(t,x,v,w) [f(w)-f(v)] \, dw. \nonumber
\end{gather}
The kernel $K$ can be any measurable function which is symmetric in $v$ and $w$ and which satisfies a coercivity bound,
\begin{equation}\label{eq:kappa_bound} \begin{gathered}
K(t,x,v,w) = K(t,x,w,v), \qquad K(t,x,v,v+w) = K(t,x,v,v-w) \\
\indic{|v-w| \leq 6} \frac{1}{\kappa} |v-w|^{-(n+2s)} \leq K(t,x,v,w) \leq \kappa |v-w|^{-(n+2s)}
\end{gathered}\end{equation}
for some constants $0 < s < 1$ and $\kappa > 1$.  The function $a$ is a source term we take to be in some Lebesgue space, the variables $t$, $x$ and $v$ are taken in $\R$, $\R^n$, and $\R^n$ respectively, and we restrict ourselves to the case $2s < n$. The integral defining $\L$ is taken in the principle value sense.  

These models are used extensively in nuclear- and astro-physics (c.f. Zaslavsky \cite{z}, Goychuk \cite{g}, and Haubold and Mathai \cite{hm}) to model the behavior of neutral particles moving through a plasma (c.f. Larsen and Keller \cite{lk}).  They can also model two-species particle fields wherein the test particles are of a very dilute species (\cite{g}).   The theory of anomalous diffusion (Mellet \cite{mellet} and Mellet, Mischler, and Mouhot \cite{mmm}) derives the small-mean-free-path limit of fractional kinetic equations such as \eqref{eq:main} and shows that these equations represent the mesoscopic behavior of fat-tailed equilibrium distributions.  These fat-tailed distributions appear in physical observations from astrophysics (\cite{lk} and Mendis and Rosenberg \cite{mr}).  

One notable special case of \eqref{eq:main} is the fractional kinetic Fokker-Planck Equation, corresponding to $\L = (-\Laplace_v)^s$ or equivalently to a homogeneous kernel $K(t,x,v,w) = C_{n,s} |v-w|^{-n-2s}$.  The (local) kinetic Fokker-Planck Equation is obtained in the limit $s \to 1$, correpsponding to $\L = - \Laplace_v$.  

If we think of $f$ as a density function for a collection of particles, with $t$, $x$, and $v$ being time, space, and velocity respectively, then the equation \eqref{eq:main} states that these particles move freely through space with their velocities changing in a stochastic manner.  If the velocity of a given particle varied according to the Weiner process, then $f$ would obey a (local) kinetic Fokker-Planck Equation.  However, when the velocity of each particle varies according to a Levy process (without drift), the density function obeys \eqref{eq:main}.  A Levy process, unlike the Weiner process, allows individual particles to change velocity suddenly and discontinuously, which better approximates the effect of elastic collisions.  

Another important model from the statistical mechanics of particles is the Boltzmann Equation
\[ \kinet f = Q(f,f). \]
In the non-cutoff case, the Boltzmann Equation sometimes enjoys a regularization effect similar the fractional Fokker-Planck equation (Alexandre, Morimoto, Ukai, Xu, and Yang \cite{amuxy}).  Our equation \eqref{eq:main} is closely related to the linear approximation of the bilinear collision operator $Q(\cdot,\cdot)$.  If the mass, energy, and entropy of a solution are assumed to be uniformly bounded, then regularization due to hypoellipticity is observed for the Boltzmann Equation (Imbert and Silvestre \cite{is}), and also for the closely related Landau Equation (Henderson and Snelson \cite{hs}, Cameron, Silvestre, and Snelson \cite{css}).  Note that \cite{is} rewrites the Boltzmann equation in the form \eqref{eq:main}, but with kernel satisfying weaker constraints than \eqref{eq:kappa_bound}.  Their regularity results are discussed below.  The most important assumption these papers require is that the mass is bounded away from the vacuum, which is connected to the coercivity of the collision operator.  In \cite{hst}, Henderson, Snelson, and Tarfulea show that this assumption really does hold for the Landau Equation.  See Mouhot \cite{m} for a thorough review of the current state of research on this front.  


Equation~\eqref{eq:main} is a typical hypoelliptic equation.  Although regularization of the integral operator happens only in $v$, we will gain regularity in $t$, $x$ thanks to the mixing property of the transport operator.  This is reminiscent of the hypoelliptic theory based on $C^\infty$ of H\"{o}rmander \cite{h} and Kolmogorov.  Averaging lemmas such as \cite{glps} (Golse, Lions, Perthame, Sentis) can be seen as an $H^s$ theory of hypoellipticity.  

This $H^s$ theory has already been applied specifically to the nonlocal kinetic Fokker-Planck Equation.  Lerner, Morimoto, and Pravda-Starov \cite{lmp} showed that solutions to certain fractional kinetic equations are in a Sobolev space $H^\sigma$ in all three variables.  This result was inspired by the work on hypoelliptic equations by Bouchut \cite{bouchut}, which is discussed in more detail below.  The precise amount of Sobolev regularity is improved and expanded upon, for example, by Morimoto and Xu \cite{mx} and by Li \cite{l}.  In fact, \cite{mx} obtains $C^\infty$ solutions in the case of no source term and $\L$ a specific operator similar to $(-\Laplace)^s$.  

This paper extends a $C^\alpha$ hypoellipticity theory, as was first introduced for kinetic Fokker-Planck by Golse, Imbert, Mouhot, and Vasseur \cite{gimv}.  They show that solutions to the (local) kinetic Fokker-Planck Equation
\[ \kinet f = \Laplace_v f \]
are H\"{o}lder continuous.  In \cite{im}, Imbert and Mouhot show that, for certain initial data, the nonlinear Fokker-Planck Equation has smooth solutions for all time.  They utilize the H\"{o}lder continuity of \cite{gimv}, as well as a Schauder-type estimate.  In \cite{is}, Imbert and Silvestre obtain H\"{o}lder continuity for a class of nonlocal kinetic Fokker-Planck-type equations with operators $\L$ more general than those considered in the present paper, and with uniformly bounded source terms.  

The seminal work on averaging lemmas is by Golse, Lions, Perthame and Sentis in 1988 \cite{glps}, which shows that solutions to $\kinet f = g$ have their weighted velocity averages $\rho[f] = \int \eta f dv$ in $H^{1/2}$, assuming $f$ and $g$ are in $L^2$.  This result had precursers in \cite{gps} and Agoshkov \cite{a}.  Many results followed, see for example DiPerna, Lions, and Meyer \cite{dlm} and DeVore and Petrova \cite{dp}, which show various levels of regularity for $\rho[f]$ assuming different regularity measures of $f$ and $g$.  

Notable in the history of averaging lemmas is \cite{bouchut}, which showed that if $f$ is regular in $v$ (in the Sobolev sense) then not only is $\rho[f]$ regular but so is $f$ itself.  This powerful result was followed by generalizations in \cite{lmp}, \cite{mx}, and \cite{l} which are especially relevant to \eqref{eq:main}.  We've used these results to establish the regularity needed to justify our calculations, as explained in Section~\ref{sec:intro}, but we do not rely on their quantitative estimates.  

Instead, the primary averaging lemma that we utilize is by Bezard \cite{bezard}.  Like Golse et al. but unlike Bouchut, this lemma gives regularity only for the density $\rho[f]$.  Bezard requires only that $f$ and $g$ lie in a negative Sobolev space $H^{-s}_v$, which gives us plenty of flexibility.  

Our proof follows the De Giorgi method, pioneered by De Giorgi in \cite{dg} (c.f. also Vasseur \cite{v}, \cite{sv}, Caffarelli and Vasseur \cite{cv}, Caffarelli, Chan, and Vasseur \cite{nio}, and \cite{gimv}).  We are particularly inspired by \cite{gimv}, which applies De Giorgi's method to a kinetic equation, and \cite{nio}, which applies the method to a nonlocal integro-differential operator.  

For two functions $f,g \in H^s(\R^n)$, and $t \in \R$ and $x \in \R^n$, define the bilinear operator
\[ B_{t,x}(f,g) = B(f,g) := \frac{1}{2}\int K(t,x,v,w) [f(w)-f(v)] [g(w)-g(v)] \, dw dv, \]
and note that
\begin{align*}
\int g(v)\L(f)(v) \,dv &= \int g(v) \int K [f(w)-f(v)]dw \, dv 
\\ &= \iint K [f(w)-f(v)] g(v) \,dwdv
\\ &= \frac{1}{2} \paren{\iint K [f(w)-f(v)]g(v) \, dwdv + \iint K [f(v)-f(w)] g(w) \,dvdw}
\\ &= - \frac{1}{2} \paren{\iint K [f(w)-f(v)][g(w)-g(v)] \,dwdv}
\\ &= - B_{t,x}(f,g).
\end{align*} 

We call $f \in L^2(Q;H^s(\R^n))$ a weak solution to \eqref{eq:main} on a domain $Q \subseteq \R \times \R^n$ when
\[ - \iiint f \kinet \phi \,dvdxdt = - \iint B(f,\phi) \,dxdt + \iiint a \phi \,dvdxdt \qquad \forall \phi \in L^2(Q; H^s(\R^n)). \]

Our main theorem is 
\begin{theorem}[Main theorem] \label{thm:main}
Given constants $s \in (0,1)$, $\kappa > 1$, and $2s < n \in \N$, there exist exponents $\alpha \in (0,1)$ and $\sourceexp_0 > 2$ such that for any open set $\Omega \subseteq \R^n$, $T > 0$, constant $\sourceexp_0 < r \leq \infty$, and source term
\[ a \in L^\sourceexp([0,T] \times \Omega \times \R^n) \cap L^2([0,T] \times \Omega \times \R^n), \]
there exists a constant such that the following is true:

If 
\[ f \in L^\infty([0,T)\times\Omega \times\R^n) \cap L^2([0,T)\times \Omega; H^s(\R^n)), \]
is a weak solution to \eqref{eq:main} subject to \eqref{eq:kappa_bound}, then $f$ is in $C^\alpha((0,T)\times\Omega\times\R^n)$.  

\begin{sloppypar}
Morover, for any $0 < \bar{T} < T$ and any compact set $\bar{\Omega} \subset \Omega$, there exists a constant $C = C(n,s,\kappa,\Omega,\bar{\Omega},T,\bar{T}) > 0$ independent of $f$ such that the following bound holds:
\[ \norm{f}_{C^\alpha\paren{[\bar{T},T]\times\bar{\Omega}\times B_1}} \leq C \paren{\norm{f_0}_{L^\infty\paren{[0,T]\times\Omega\times \R^n}} + \norm{a}_{L^\sourceexp\paren{[0,T]\times\Omega\times \R^n}}}.  \]
\end{sloppypar}
\end{theorem}

Although the assumption \eqref{eq:kappa_bound} on the kernel is a natural one for studying absolutely continuous kernels from an energy perspective, it is too strict to apply to e.g. the Boltzmann equation because the collision kernel may not be absolutely continuous or symmetric.   As a result, in the case $a \in L^\infty$, our result is included in the result of \cite{is}.  Their proof does not use a averaging lemma, instead utilizing a careful study of the Green's function for the fractional Kolmogorov equation.  They employ a Krylov approach to obtain a weak Harnack inequality. The advantage of our stronger assumtptions on the kernel is that our proof can be entirely energy based, which allows us to consider source terms which are not uniformly bounded.  We are also able to take a unified variational approach to the cases $s < 1/2$ and $s > 1/2$ by adapting the technique of \cite{nio} to the kinetic context.  

The assumption that solutions are in $L^\infty$ will hold in particular when the initial data and source term are both in $L^\infty$.  In such a case, we could obtain a maximum principle by computing $\ddt \iint (f-C-t\norm{a}_\infty)_+^2 dvdx$.  

With arbitrary source term, a more robust $L^\infty$ bound can sometimes be obtained by adapting Proposition~\ref{thm:DG1} below.  As stated, this proposition requires an assumption of uniformly bounded growth for large values of $v$, to avoid interactions between high-velocity particles and the boundary $\del \Omega$ of our spatial domain.  Though outside the scope of the present paper, this assumption could be removed with proper boundary conditions.  For example, if we take $x \in \mathbb{T}^n$ the torus, then solutions will be $L^\infty$ at any positive time.  

In the case that $K$ is homogeneous near the origin, meaning equal to $|v-w|^{-n-2s}$ for $|v-w|$ sufficiently small, we can obtain existence of an $L^2(H^s)$ weak solution from \cite{mx} Theorem 1.1 (by treating the difference between $\L f$ and $(-\Laplace)^sf$ as a source term).  When $K$ is not homogeneous near the origin, our result is an a priori estimate.  In particular, when a uniform $L^\infty$ bound exists (as discussed above), we can obtain existence of continuous solutions through the method of continuity.  


The symmetry assumption posed in \eqref{eq:kappa_bound} is actually two symmetry assumptions.  The former, $K(t,x,v,w) = K(t,x,w,v)$, is crucial to the weak formulation of the problem and hence is used throughout this paper.  The latter assumption $K(t,x,v,v+w) = K(t,x,v,v-w)$ is really only used in the proof of Lemma~\ref{thm:psi_properties}.  It is necessary because otherwise, in the case $s \geq 1/2$, the operator $\L$ might not be bounded even from $\Ctest$ to $L^\infty$.  We list here a few alternative assumptions, any one of which could replace the latter symmetry assumption of \eqref{eq:kappa_bound} with no loss of generality.  
\begin{itemize}
\item For any $C^2$ function $\phi$, $\norm{\L \phi}_\infty \leq C \norm{\phi}_{C^2}$. 
\item The parameter $s$ is strictly less than $1/2$.  
\item For any $t,x,v \in \Rall$, $\int_{B_1} w K(t,x,v,w) \,dw = 0$.  
\item The function $K(t,x,v,v+w)$ is independent of $v$. 
\end{itemize}

The lower bound on the exponent $\sourceexp$ for the source term is
\[ \sourceexp_0 = \frac{n (1+s) (n+1)}{s} \paren{2 \frac{s}{n} + \frac{1}{2} + \frac{n}{2s}}. \]
This bound is strictly greater than 2, and it is also strictly greater than $n+1+n/s$, which is the critical scaling exponent.  This lower bound may not be sharp.   

The remainder of this article is dedicated to the proof of Theorem~\ref{thm:main}.  Section~\ref{sec:preliminaries} contains a few preliminary lemmas.  Sections~\ref{sec:DG1} and \ref{sec:DG2} are dedicated to the proofs of the first and second De Giorgi lemmas, respectively.  Section~\ref{sec:Holder} combines the De Giorgi lemmas to obtain a Harnack inequality that proves Theorem~\ref{thm:main}.  

Throughout this paper, $C$ will represent arbitrary constants which may change from line to line.  A constant is called ``universal'' if it depends only on the dimension $n$, the order $s$ of the operator $\L$, and the coercivity bound $\kappa$.  

The function space $\Ctest$ contains smooth functions with compact support.  


\section{Preliminary Lemmas} \label{sec:preliminaries}

This section contains three lemmas which will be relied upon extensively in the forthcoming sections.  

The operator $\L$ behaves in many ways like the operator $-(-\Laplace_v)^{s} = - \Lambda^{2s}$.  The following lemma codifies the important similarities between the two operators, specifically the relationship between $B$ and the $H^s$ norm, and between $\L$ and the Bessell potential.  
\begin{lemma} \label{thm:B_and_Hs}
There exists a constan $C = C(n,s,\kappa)$ such that, for any function $f \in H^s(\R_n)$, we have the following bounds:
\[ \int |\Lambda^s f|^2 \,dv \leq \inf_{t,x} C\paren{B_{t,x}(f,f) + \int f^2 \,dv}, \]
and
\[ \sup_{t,x} \norm{\bessel^{-s/2} \L_{t,x} f}_{ L^2(\R^n) } \leq C \paren{ \norm{\Lambda^s f}_{L^2(\R^n)}}. \]

\end{lemma}

Since these results are true irrespective of $t$ and $x$, we will omit their mention in the sequel.  

\begin{proof}
For the first inequality, simply calculate
\begin{align*}
B(f, f) &= \iint K \bracket{f(w) - f(v)}^2 \,dwdv
\\ &\geq \frac{1}{\kappa} \iint_{|v-w| \leq 6}  \frac{\bracket{f(w) - f(v)}^2}{|v-w|^{n+2s}} \,dwdv
\\ &= \frac{1}{\kappa} \iint \frac{\bracket{f(w) - f(v)}^2}{|v-w|^{n+2s}} \,dwdv - \frac{1}{\kappa} \iint_{|v-w| \geq 6}  \frac{\bracket{f(w) - f(v)}^2}{|v-w|^{n+2s}} \,dwdv
\\ &\geq \frac{1}{\kappa} \iint \frac{\bracket{f(w) - f(v)}^2}{|v-w|^{n+2s}} \,dwdv - \frac{2}{\kappa} \int f(v)^2 {\int  \frac{\indic{|u|\geq 6}}{|u|^{n+2s}} du} \,dv - \frac{2}{\kappa} \int f(w)^2 {\int \frac{\indic{|u|\geq 6}}{|u|^{n+2s}} du} \,dw
\\ &	= C(n,s,\kappa) \int |\Lambda^s f|^2 \, dv - C'(n,s,\kappa) \int f^2 \,dv.
\end{align*}

For the second inequality, let $g$ be any function in $H^s(\R^n)$.  For $t$ and $x$ fixed, we have the following bound on inner products in $v$:
\begin{align*}
\abs{\chevron{\L f, g}}_v &= \abs{\iint [f(v+w)-f(v)][g(v+w) - g(v)] K(t,x,v,v+w) \,dwdv}
\\ &= \abs{\iint \paren{[f(v+w)-f(v)]|w|^{\frac{n+2s}{2}}}\frac{[g(v+w) - g(v)]}{|w|^{\frac{n+2s}{2}}} K \,dwdv}
\\ &\leq \paren{\iint [f(v+w)-f(v)]^2 K^2 |w|^{n+2s} \,dwdv}^{1/2} \paren{\iint [g(v+w) - g(v)]^{2} \frac{dwdv}{|w|^{n+2s}}}^{1/2}
\\ &\leq \kappa \paren{\iint [f(v+w)-f(v)]^2 \frac{dwdv}{|w|^{n+2s}}}^{1/2} \paren{\iint [g(v+w) - g(v)]^{2} \frac{dwdv}{|w|^{n+2s}}}^{1/2}
\\ &= C(n,s,\kappa) \paren{\int \abs{\Lambda^s f}^2 \,dv}^{1/2} \norm{g(t,x,\cdot)}_{H^s(\R^n)}.
\end{align*}

Therefore if $\phi$ is any $L^{2}(\R^n)$ test function, then
\begin{align*}
\chevron{\bessel^{-s/2}\L f,\phi} &= \chevron{\L f,\bessel^{-s/2}\phi}
\\ &\leq C(n,s,\kappa) \paren{\int \abs{\Lambda^s f}^2 \,dv}^{1/2}  \norm{\bessel^{-s/2}\phi}_{H^s(\R^n)}
\\ &= C(n,s,\kappa) \paren{\int \abs{\Lambda^s f}^2 \,dv}^{1/2} \norm{\phi}_{L^2(\R^n)}.
\end{align*}

The lemma follows by taking a supremum over all such $\phi$.  

\end{proof}

We now come to the energy inequality.  An inequality of this type is to be expected due to the parabolic flavor of Equation~\eqref{eq:main}, and it is in some ways the most important quality of our equation.  Notice that the inequality gives control over the regularity in $v$, but not in $t$ or $x$.  

\begin{lemma}[Energy Inequality] \label{thm:energy_inequality}
There exists a universal constant $C = C(n,s,\kappa)$ such that the following is true:

Let $T < S < 0$ be times, and let $\Omega \subseteq \R^n$ be an open region in space and $\bar{\Omega} \subseteq \Omega$ a compact subset.  Let $R > 0$ a radius and $\psi:\R^n \to \R$ a function of velocity.  Denote $Q := (T,0] \times \Omega$ and $\bar{Q} := [S,0]\times\bar{\Omega}$, and define
\begin{align*} 
\delta &:= \min\paren{|T-S|, \operatorname{dist}(\bar{\Omega}, \Omega^\mathsf{C})}.
\end{align*}

Let $f \in L^2(\Q; H^s(\R^n))$ be any weak solution to \eqref{eq:main} subject to \eqref{eq:kappa_bound} satisfying
\[ f(t,x,v) \leq \psi(v) \qquad \forall (t,x) \in Q, |v| \geq R, \] 
and denote $f_+ := \max(f-\psi,0)$ and $f_- := \max(\psi-f,0)$ so that $f = f_+ + \psi - f_-$.  

Then the following energy inequality holds:
\begin{multline*}
\iint_{\bar{Q}} B(f_+,f_+) \,dxdt - \iint_{\bar{Q}} B(f_+,f_-) \,dxdt \leq \\
\frac{C}{\delta} \bracket{ R \iint_Q \int f_+^2 \,dvdxdt + \paren{\sup_{|v|<R}\abs{\L\psi(v)}} \iint_Q \int f_+ \,dvdxdt + \norm{a}_{L^\sourceexp(Q)} \norm{f_+}_{L^{\sourceexp^\ast}(Q)} }. 
\end{multline*}
\end{lemma}

The constant $\delta$ here is the distance from $\bar{Q}$ to the parabolic boundary of $Q$.  

The quantity $B(f_+,f_+)$ is, as shown in Lemma~\ref{thm:B_and_Hs}, related to the fractional Dirichlet energy of $f_+$.  We have an additional dissipation term $-B(f_+,f_-)$ which we call the cross term.  Because $f_+$ and $f_-$ have disjoint supports,
\begin{align*}
-B(f_+,f_-) &= -\iint K [f_+(w)-f_+(v)][f_-(w)-f_-(v)]\,dwdv
\\ &= \iint K [f_+(w)f_-(v) + f_+(v)f_-(w)] \,dwdv
\\ &= 2 \iint K f_+(v)f_-(w) \,dwdv.
\end{align*} 
In particular this means the cross term is non-negative.  The cross term represents, in a sense, the energy which is lost when we localize $f$ to create $f_+$.  The bound on the cross term is critical to our proof in Section~\ref{sec:DG2} of De Giorgi's second lemma.  

\begin{remark}
The quantity $f_-$ appears on the left but not the right hand side of the energy inequality.  This means in particular that the growth and decay of any solution to \eqref{eq:main} is constrained by the local behavior alone.  
\end{remark}


\begin{proof}
Define $\phi:\R^n \to [0,1]$ a function which equals 1 on $\bar{\Omega}$, which is supported on $\Omega$, and which is Lipschitz with constant $\norm{\phi}_{C^1} \leq 2 \delta\n$.  

Multiplying the left side of Equation~\eqref{eq:main} by the quantity $\phi^2 f_+$, we see that
\begin{align*}
\phi^2 f_+ \kinet f &= \phi^2 f_+ \kinet(f_++\psi-f_-)
\\ &= \phi^2\frac{1}{2} \kinet f_+^2 + \phi^2 f_+ \kinet \psi - \phi^2 f_+ \kinet f_-
\\ &= \frac{\phi^2}{2} \kinet f_+^2
\end{align*}
because $\psi$ is independent of $x$ and $t$, and $f_+$ and $f_-$ have disjoint supports.  

Since $f_+ \in L^2(H^s)$, we can multiply Equation~\eqref{eq:main} by $2 \phi^2 f_+$ and integrate with respect to $x$ and $v$ to obtain
\begin{multline*}
\ddt \iint (\phi f_+)^2 \,dvdx - \iint f_+^2 v\cdot\grad_x(\phi^2) \,dvdx = -2 \int \phi^2 B(f_+, f_+ + \psi - f_-) dx + 2 \iint \phi^2 a f_+ \,dvdx
\\ = -2 \int \phi^2 B(f_+, f_+) dx - 2 \iint \phi^2 f_+ \L\psi \,dvdx + 2 \int \phi^2 B(f_+, f_-) dx + 2 \iint \phi^2 a f_+ \,dvdx.
\end{multline*}

For any $S \leq \tau \leq T$, we integrate this equality from $\tau$ to $0$ in time and then rearrange to obtain
\begin{multline*} 
\iint (\phi f_+(0))^2 \,dvdx + 2\int_\tau^0\!\!\!\!\int \phi^2 B(f_+,f_+) \,dxdt - 2\int_\tau^0\!\!\!\!\int \phi^2 B(f_+,f_-) \,dxdt \\
=\!\!\! \int_\tau^0\!\!\!\!\iint \paren{v\!\cdot\!\grad_x \phi^2} f_+^2 \,dvdxdt + 2\!\!\int_\tau^0\!\!\!\!\iint \phi^2 {\L(\psi)} f_+ \,dvdxdt + \int_\tau^0\!\!\!\!\iint \phi^2 a f_+ \,dvdxdt + \iint (\phi f_+(\tau))^2 \,dvdx. 
\end{multline*}

In particular,
\begin{multline*} 
2\int_T^0\!\!\!\!\int \phi^2 B(f_+,f_+) \,dxdt - 2\int_T^0\!\!\!\!\int \phi^2 B(f_+,f_-) \,dxdt \\
\leq \int_S^0\!\!\!\!\iint \abs{v\cdot\grad_x \phi^2} f_+^2 \,dvdxdt + 2\int_S^0\!\!\!\!\iint \phi^2 \paren{\abs{\L(\psi)}+|a|} f_+ \,dvdxdt + \iint (\phi f_+(\tau))^2 \,dvdx. 
\end{multline*}

Now only one term depends on $\tau$.  If we take the average value over $\tau \in [S,T]$ for both sides of the inequality, we obtain
\begin{multline*} 
2\int_T^0\!\!\!\!\int \phi^2 B(f_+,f_+) \,dxdt - 2\int_T^0\!\!\!\!\int \phi^2 B(f_+,f_-) \,dxdt \\
\leq \int_S^0\!\!\!\!\iint \abs{v\cdot\grad_x \phi^2} f_+^2 \,dvdxdt + 2\int_S^0\!\!\!\!\iint \phi^2 \paren{\abs{\L(\psi)}+|a|} f_+ \,dvdxdt + \frac{1}{|S-T|}\int_S^T\!\!\!\!\iint (\phi f_+)^2 \,dvdxdt. 
\end{multline*}

Our energy inequality follows.  
\end{proof}

The classical technique to localize a solution to a PDE is multiplication by a compactly supported cutoff function.  This allows us to disregard the behavior of the solution outside a specified region, while the localized function usually solves the original PDE, modulo some sort of error term.  One should not expect this technique to work for nonlocal PDE; the far-away behavior of the solution cannot be completely disregarded.  

Instead, we must localize by a ``soft cutoff,'' which is a fixed function $\psi$ that vanishes in a specified local region but grows without bound outside that region.  We have already seen soft cutoffs used in the statement and proof of Lemma~\ref{thm:energy_inequality} just above.  
%
%

Throughout the following sections, we will utilize a few different soft cutoff functions.  We will define all of our soft cutoff functions here and list all their relevant properties, then refer back to this lemma as we use them.  These functions $\psi^1$ and $\psi_\theta$ are tailored to the required assumptions of Lemmas~\ref{thm:DG1} and \ref{thm:DG2} respectively.  They also must have certain specific relationships with eachother in order to prove Lemma~\ref{thm:oscillation}, which is why we prefer to construct them here all at once.  

\begin{lemma} \label{thm:psi_properties}
Let $s \in (0,1)$ and $2s < n \in \N$ be specified constants.  There exists a function $\psi^1:\R^n \to \R^+$ and a family of functions $\psi_\theta:\R^n \to \R^+$ indexed by $\theta \in (0,1)$ with the following properties:

\begin{enumerate}[(i)]
\item\label{L_of_psi} There exists a constant $C_\psi$ such that for all $v \in \R^n$
\[ \sup_{t,x} \abs{\L_{t,x}\psi^1(v)} \leq C_\psi, \qquad \sup_{t,x} \abs{\L_{t,x}\psi_\theta(v)} \leq C_\psi, \]
and for all $|v| \leq 3$
\[ \sup_{t,x} \abs{\L_{t,x}\psi_\theta(v)} \leq C_\psi \theta^{3s/2}. \]

\item\label{psi_vanish_on_balls} For $|v| \leq 1$, 
\[ \psi^1(v)=0 \]
 and for $|v| \leq \theta\n$, 
 \[ \psi_\theta(v)=0. \]

\item\label{psi_increasing} For any $\theta < \vartheta$, and for all $v \in \R^n$
\[ \psi_\theta(v) \leq \psi_\vartheta(v) \leq \psi^1(v). \]

\item\label{psi_geq_one_plus_psi} For all $|v| \geq 2$, for any $\theta \in (0,1)$,
\[ 1 + \psi_\theta(v) \leq \psi^1(v). \]

\item\label{psi_scaled_inequality} For each $\theta$, there exists $\eps_0 = \eps_0(s,\theta)$ such that $\eps < \eps_0$ implies that for all $|v| > \eps\n$,
\[ \psi_\theta\paren{v} \geq 2\psi_\theta(\eps v) + 2. \] 
\end{enumerate}
\end{lemma}

\begin{proof}
First define a function $g:[0,\infty) \to [0,\infty)$ such that, for all $x > 1$,
\[ g(x) = x^{s/2} \] 
but $g(0) = g'(0) = 0$, and in the interval $[0,1]$ let $g$ be defined so that $g$ is smooth and non-decreasing, and $g(x) \leq x^{s/2}$.  

Next define functions $g_r$ for each $r > 0$ by
\[ g_r(x) = \begin{cases}
0 & x < r  \\
g(x-r) & x \geq r.
\end{cases}\] 

Then $g_r$ is pointwise-decreasing in $r$ and both $\norm{g_r''}_{L^\infty}$ and the H\"{o}lder semi-norm $\norm{g_r}_{\dot{C}^{s/2}}$ are finite and independent of $r$.  

We'll define
\[ \psi_\theta(v) := g_{\theta\n}(|v|). \]

Let $C_1 > 1$ be a constant large enough that for any $\theta \in (0,1)$, for all $|v| \geq 2$
\[ 1 + \psi_\theta(v) \leq C_1 g_1(|v|). \]
Then define
\[ \psi^1(v) = C_1 g_1(|v|). \]

Properties \eqref{psi_vanish_on_balls}, \eqref{psi_increasing}, and \eqref{psi_geq_one_plus_psi} all follow immediately from the construction.  Notice also that all of these functions have uniformly bounded second derivatives and uniformly bounded $\dot{C}^{s/2}$ semi-norms.  

Let $\psi$ be either $\psi^1$ or any of the $\psi_\theta$, and let $v \in \R^n$ be chosen.  We wish to calculate $\L\psi(v)$, so let us break up the defining integral into the ``near" part and the ''far" part.  
\[ \L \psi(v) = \int_{|w|<1} K(v,v+w)[\psi(v+w)-\psi(v)] \,dw + \int_{|w|\geq 1} K(v,v+w)[\psi(v+w)-\psi(v)] \,dw. \]

For the near part, we utilize the fact that $\psi$ is smooth with bounded second derivative.  
We apply Taylor's theorom to find that
\[ \psi(v+w)-\psi(v) = D\psi(v)\cdot w + D^2\psi(u) w \otimes w \]
for some $u$ on the line segment between $v$ and $v+w$.  By the symmetry \eqref{eq:kappa_bound} of $K$,
\[ \int_{|w|<1} K(v,v+w) D\psi(v)\cdot w \,dw = 0. \]  
Note that this integral must be understood in the principal value sense.  

The remainder is 
\[ \int_{|w|<1} K(v,v+w)D^2\psi(u) w \otimes w \,dw \leq C \kappa \int_{|w|<1} \frac{|w|^2}{|w|^{n+2s}} \,dw, \]
with $C$ here being the bound on $\norm{D^2\psi}_\infty$ which is independent of $\psi$.  
Since $n + 2s - 2 < n$, the integral is finite.  

Notice that if $\psi = \psi_\theta$ with $\theta < 1/4$ and if $|v|\leq 3$ then the near part of the integral is in fact zero.  

For the far away part, we utilize the fact that $\psi$ is H\"{o}lder continuous in $\dot{C}^{s/2}$ and estimate 
\[ \abs{\psi(v+w) - \psi(v)} \leq C|w|^{s/2} \]
with $C$ independent of $\psi$.  
The integral of the far away part becomes
\[ \int_{|w|\geq 1} K(v,v+w)[\psi(v+w)-\psi(v)] \,dw \leq C \kappa \int_{|w|\geq 1} \frac{|w|^{s/2}}{|w|^{n+2s}} \,dw. \]
Since  $n+2s-\frac{s}{2} > n$, the integral is finite.  

In the case $\psi = \psi_\theta$ with $\theta < 1/4$ and $|v|\leq 3$, $\psi(v)=0$ so we can make the stronger estimate 
\[ \abs{\psi(v+w) - \psi(v)} \leq g_{\theta\n}(|w| + 3) \leq \max(|w|+3-\theta\n,0)^{s/2}. \]
The integral of the far away part becomes
\[ \int_{|w|\geq 1} K(v,v+w)[\psi(v+w)-\psi(v)] \,dw \leq  \kappa \int_{|w|\geq \paren{\theta\n-3}} \frac{(|w|+3-\theta\n)^{s/2}}{|w|^{n+2s}} \,dw \leq C \int_{|w| > \frac{\theta\n}{4}} \frac{dw}{|w|^{n+\frac{3}{2}s}}. \]
This integral is proportional to $\theta^{3s/2}$.  The property \eqref{L_of_psi} follows.  

All that remains is to show \eqref{psi_scaled_inequality}, so fix some value of $\theta$.  We'll show the equivalent claim 
\begin{equation}\label{altered_scaled_inequality}
\psi_\theta(v/\eps) \geq 2 \psi_\theta(v) + 2 \qquad \forall |v| \geq 1. 
\end{equation}
For $|v| \geq \theta\n+1$ and any $0 < \eps < 1$, we can say 
\[ \psi_\theta(v/\eps) = (|v|/\eps-\theta\n)^{s/2} \geq (|v|/\eps - \theta\n/\eps)^{s/2} = \eps^{-s/2} (|v| - \theta\n)^{s/2}. \]
There exists $0 < \eps_1 < 1$ and $r_1 > \theta\n+1$ so that if $\eps < \eps_1$ and $|v| \geq r_1$ then
\[ \eps^{-s/2} (|v| - \theta\n)^{s/2} \geq 2 \psi_\theta(v) + 2. \]
Now take $\eps_0 < \eps_1$ small enough that $\psi_\theta(1/\eps_0) \geq 2 \psi_\theta(r_1) + 2$.  Now for $1 \leq |v| \leq r_1$ the inequality \eqref{altered_scaled_inequality} holds because
\[ \psi_\theta\paren{\frac{v}{\eps}} \geq \psi_\theta(1/\eps_0) \geq 2 \psi_\theta(r_1) + 2 \geq 2 \psi_\theta(v) + 2, \]
and for $|v| > r_1$ it holds by construction of $r_1$.  This proves property \eqref{altered_scaled_inequality}.  
\end{proof}


\section{First De Giorgi Lemma}\label{sec:DG1}

In this section we will prove De Giorgi's first lemma, which states that if a function solving \eqref{eq:main} is bounded in some region in an integral sense, then it is pointwise bounded in a smaller region.  

The function $\psi^1$ in the statement of this lemma is defined in Lemma~\ref{thm:psi_properties}.  

\begin{proposition}[De Giorgi's First Lemma]\label{thm:DG1}
There exists a universal constant $\delta_0 > 0$ such that the following is true:

For any $f \in L^2([-2,0]\times B_2; H^s(\R^n))$ a weak solution to \eqref{eq:main} subject to \eqref{eq:kappa_bound} with source term $\norm{a}_{L^\sourceexp([-2,0]\times B_2 \times\R^n)} \leq 1$, if
\[ f(t,x,v) \leq \psi^1(v) \qquad \forall x \in B_2, t \in [-2,0], |v| \geq 2 \]
holds and 
\[ \iiint_{[-2,0]\times B_2\times B_2} \max(f-\psi^1,0)^2 \, dvdxdt \leq \delta_0 \]
holds, then
\[ f(t,x,v) \leq \frac{1}{2} \qquad \forall x \in B_1, t \in [-1,0], v \in B_1. \]
\end{proposition}

As in most De Giorgi-style proofs, we take a sequence of cutoffs of our function and show that their $L^2$ norm tends to zero.  We show this by producing a non-linear recursive inequality.
The key to the proof is the inequality \eqref{recursive_inequality}, which is located at the end of the second step.  This inequality tells us that our function cannot have very bad singularities, because any singularity which is $L^2$ integrable is also $L^q$ integrable for some specific $q > 2$.  Classically such an inequality is produced using the energy inequality and Sobolev embedding, but in this case we will also require an averaging lemma.  

Our proof will proceed in three steps.  In the first step, we will apply the averaging lemma to our cutoff function to show that it has higher integrability in the $t$ and $x$ variables.  Actually we will apply the averaging lemma to a barrier function, because our solution itself has certain negative measures in its derivatives.  This is fine, since higher integrability for the barrier function trivially implies higher integrability for the original function.  In the second step, we will obtain higher integrability in the $v$ variable using the usual technique (with the energy inequality and Sobolev embedding).  Then we use Riesz-Thorin interpolation to combine our integrability in $t$, $x$ and $v$.  In the third and final step, we produce the standard nonlinear recursion and argue that our cutoffs tend to zero in the limit. 

\begin{proof}
We begin by specifying the sequence of cutoff functions.  For $k \in \N$, consider soft cutoffs
\[ \psi_k := \psi^1 + \frac{1}{2} - 2^{-k-1} \]
so that $\psi_0 = \psi^1$ and in the limit $\psi_\infty = \psi^1 + \frac{1}{2}$.  Then we have a sequence of cutoff functions
\[ f_k := \max(f - \psi_k,0). \]

We'll make frequent use of the fact that for any $k$,
\begin{equation}\label{indicator_bound} 
\indic{f_k > 0} \leq 2^{k+1} f_{k-1}. 
\end{equation}

We also must specify a sequence of space-time regions.  Define
\[ T_k := -1-2^{-k}, \qquad B^k := \{x \in \R^n: |x| \leq 1 + 2^{-k}\}, \qquad Q_k := [T_k,0]\times B^k \]
so that $Q_0 = [-2,0]\times B_2$ and in the limit $Q_\infty = [-1,0]\times B_1$.  Notice that the distance from the interior of $B^k$ to the boundary of $B^{k-1}$ is $2^{-k}$, and that $T_k - T_{k-1} = 2^{-k}$.  

For brevity, we will use $\int_k$ to denote an integral with bounds $[T_k,0]$ or $B^k$ or $Q_k$, as shall be clear from context.  We also frequently will use $C^k$ to mean $\bracket{C(n,s,\kappa)}^k$, a quantity which grows geometrically in $k$ for $n$, $s$, and $\kappa$ held constant.  
%

\step{Higher integrability in $t,x$}

\newcommand{\etak}{\eta_{k,\eps}}

Define $\etak$ a smooth function which is supported on $[T_{k-1},0]$ and equal to 1 on $[T_k, -\eps]$.  Then define $\mu_\eps(t) = \indic{[-\eps,0]}\del_t \etak$ the derivative of $\etak$ near 0, and assume without loss of generality that $\mu_\eps \leq 0$.  The derivative of $\etak$ will be bounded uniformly in $\eps$ \emph{except} for the blowup near 0 which is captured by $\mu_\eps$.  In symbols, $\sup_\eps \norm{\del_t \etak - \mu_\eps}_\infty \leq C^k$.  

In addition, let $\phi_k(x)$ be a smooth function supported on $B^{k-1}$ and equal to 1 on $B^k$, with derivative $\norm{\grad_x \phi_k}_\infty \leq C^k$.

We want to apply the averaging lemma to $\etak \phi_k f_k$, so let's apply the transport operator to this function.  
\begin{gather*}
\begin{aligned}
\kinet (\etak \phi_k f_k) &= f_k \kinet (\etak\phi_k) + \etak\phi_k \kinet f_k
\\ &= f_k \kinet (\etak \phi_k) + \etak \phi_k \indic{f > \psi_k} \kinet (f - \psi_k)
\\ &= f_k \kinet (\etak \phi_k) + \etak \phi_k \indic{f > \psi_k} \L f + \etak \phi_k \indic{f > \psi_k} a
\end{aligned}
\\ \qquad \quad \quad
= f_k \kinet (\etak \phi_k) + \etak \phi_k \indic{f > \psi_k} \L \psi_k + \etak \phi_k \indic{f > \psi_k} a + \etak \phi_k \indic{f > \psi_k} \L (f - \psi_k).
\end{gather*}

By a well known pointwise inequality (c.f. Caffarelli and Sire \cite{cs}),
\[ \indic{f > \psi_k} \L (f - \psi_k) \leq \L f_k. \]
Also $\mu_\eps \leq 0$.  Therefore if we define
\[ F_k := \etak f_k (v\cdot\grad_x \phi_k) + \phi_k f_k (\del_t \etak - \mu_\eps) + \etak \phi_k \indic{f > \psi_k} \L \psi_k + \etak \phi_k \indic{f > \psi_k} a, \]
then
\[ \kinet ( \etak \phi_k f_k) \leq F_k + \L(\etak \phi_k f_k). \]

The source term $F_k$ is in $L^2(\Rall)$.  From \eqref{indicator_bound}, Lemma~\ref{thm:psi_properties} property \eqref{L_of_psi}, and the definitions of $\phi_k$, $\etak$ and $\mu_\eps$, 
\begin{align}
\iiint\displaylimits_{\Rall} F_k^2 &\leq \iiint_{k-1} \bracket{\etak^2 (v\cdot\grad_x\phi_k)^2 + \phi_k^2 (\del_t \etak - \mu_\eps)^2} f_k^2 + \iiint_{k-1} (\etak \phi_k)^2 \bracket{(\L \psi_k)^2 + a^2} \indic{f_k > 0} \nonumber
\\ &\leq C^k \iiint_{k-1} f_k^2 + C^k \iiint_{k-1} f_{k-1}^2 + C^k \paren{\iiint_{k-1} f_{k-1}^2}^{1-\frac{2}{\sourceexp}} \nonumber
\\ &\leq C^k \paren{\iiint_{k-1} f_{k-1}^2 }^{1-\frac{2}{\sourceexp}}. \label{Fk_bound}
\end{align}

Because the averaging lemma requires equality, not the inequality that we have, we'll construct a barrier function $g_k$.  Define $g_k$ as some solution to the PDE
\begin{equation}\begin{cases} \label{gk_PDE}
\kinet g_k = F_k + \L g_k & \forall t,x,v \in (T_{k-1},\infty) \times \R^n \times \R^n \\ 
g_k = \etak \phi_k f_k = 0 & t = T_{k-1} \\
g_k = 0 & t < T_{k-1}.
\end{cases}\end{equation}

Since $F_k \in L^2(\Rall)$, a solution $g_k \in L^2_\loc([0,\infty)\times \R^n; H^s(\R^n))$ exists by \cite{mx} (see Section~\ref{sec:intro} for more detail).  

Moreover, $g_k \geq \etak \phi_k f_k \geq 0$ by a maximum principle: the function $\max(\etak \phi_k f_k - g_k,0)$ is a subsolution to $\kinet h = \L h$ so it has non-increasing energy, and it vanishes at $t = T_{k-1}$ so it must be identically zero.  

We'll now produce some bounds on $g_k$.  Take the PDE \eqref{gk_PDE} and multiply it by $g_k$, then integrate over $x \in \R^n$, $v \in \R^n$.  
\[ \ddt \frac{1}{2}\iint g_k^2 \,dvdx = - \int B(g_k,g_k) \,dx + \iint g_k F_k \,dvdx. \]
Now applying Lemma~\ref{thm:B_and_Hs} and H\"{o}lder's inequality,
\begin{equation}\label{energy_inequality_for_g}
\ddt \frac{1}{2}\iint g_k^2 \,dvdx + \frac{1}{\kappa} \iint \abs{\Lambda^s g_k}^2 \,dvdx \leq C \iint g_k^2 \,dvdx + \frac{1}{2}\iint F_k^2 \,dvdx. 
\end{equation}

If we define
\[ G(t) = \iint_{\R^n \times \R^n} g_k^2(t) \,dvdx \]
we see from \eqref{energy_inequality_for_g} that $G$ satisfies
\[ \ddt G(t) \leq C G(t) + \iint F_k^2(t). \]
Also, by construction, $G(T_{k-1}) = 0$.  Thus by Gronwall's inequality, for all $t > T_{k-1}$:
\begin{align*} 
G(t) &\leq e^{C(t-T_{k-1})} \int_{T_{k-1}}^t \iint F_k^2(\tau) \,dvdx \,d\tau
\\ &\leq e^{C(t-T_{k-1})} \iiint_{\Rall} F_k^2 \,dvdxd\tau.
\end{align*}

This means that for any compact interval $K$ in $\R$,
\begin{equation}\label{gk_bound}
\norm{g_k}_{L^\infty(K; L^2(\R^n\times \R^n))} \leq C_K \norm{F_k}_{L^2(\Rall)}. 
\end{equation}
Armed with this inequality, and the fact that $\del_t \chi_K$ is in the dual space of $L^\infty(t)$, we integrate \eqref{energy_inequality_for_g} over $K$:
\begin{align} 
\iiint\displaylimits_{K \times \R^n \times \R^n} \abs{\Lambda^s g_k}^2 \,dvdxdt &\leq C(n,s,\kappa) \paren{ \iiint_{K \times \R^n \times \R^n} g_k^2 + \iiint_{K \times \R^n \times \R^n} F_k^2 \,dvdxdt + \iiint_{\R \times \R^n \times \R^n} g_k^2 \del_t \chi_K} \nonumber
\\ &\leq C_K \iiint_{\Rall} F_k^2 \,dvdxdt. \label{Hs(gk)_bound}
\end{align} 
%
%
%
%

We can now apply Lemma~\ref{thm:avg_lemma}, the Averaging Lemma, to $g_k$.  Let $\eta(v)$ be a $\Ctest(\R^n)$ function which is identically 1 on $v \in B_2$ and non-negative for all $v$, and choose any set, for example $[-3,1]\times B_3$, which compactly contains $[-2,0]\times B_2$.  The lemma yields that
\begin{equation*}
\norm{\int \eta g_k \,dv}_{H^\beta([-2,0]\times B_2)} \leq C\paren{\norm{g_k}_{L^2([-3,1]\times B_3 \times \R^n)} + \norm{\bessel^{-s/2}\paren{F_k + \L g_k}}_{L^2([-3,1]\times B_3 \times \R^n)}} 
\end{equation*}
with $\beta = (2(1+s))\n < 1$.  

Therefore, by the bounds \eqref{Fk_bound} and \eqref{gk_bound}, and by Lemma~\ref{thm:B_and_Hs} and the bound \eqref{Hs(gk)_bound},
\begin{equation}\label{bezard_result}
\norm{\int \eta g_k \,dv}_{H^\beta([-2,0]\times B_2)} \leq C^k \paren{\iiint_{k-1} f_{k-1}^2}^{\frac{1}{2} - \frac{1}{\sourceexp}}. 
\end{equation}

Define $\txexp$ by
\[ \frac{1}{\txexp} = \frac{1}{2} - \frac{\beta}{n+1} = \frac{1}{2} - \frac{1}{2(1+s)(n+1)} \in (0,1/2). \]
By Sobolev embedding,
\begin{equation}\label{sobolev}
\norm{ \int \eta g_k \,dv}_{L^\txexp(t,x)} \leq C \norm{ \int \eta g_k \,dv}_{H^\beta(t,x)}. 
\end{equation} 

Since $f_k$ is supported where $\eta \equiv 1$, the integral $\int \eta f_k \,dv$ is just the $L^1(v)$ norm of $f_k$.  Recall also that $\etak \phi_k f_k \leq g_k$.  Therefore we can bound the $L^{\txexp,\txexp,1}$ norm of $f_k$:
\begin{align*}
\int_{T_k}^{-\eps} \int_{B^k} \paren{ \int f_k \,dv}^\txexp \, dxdt 
 &\leq \iint \paren{ \int \eta \bracket{\etak \phi_k f_k} \,dv}^\txexp \, dxdt
\\ &\leq \iint \paren{ \int \eta g_k \,dv}^\txexp \, dxdt
\end{align*}


Since this inequality is true for all $\eps$, we can chain it with \eqref{bezard_result} and \eqref{sobolev} to conclude that
\begin{equation}\label{t_and_x_integrability}
\norm{f_k}_{L^{\txexp,\txexp,1}(Q_k \times \R^n)} \leq C^k \norm{f_{k-1}}_{L^2(Q_{k-1} \times \R^n)}^{1-\frac{2}{\sourceexp}}. 
\end{equation}

\step{Higher integrability in all three variables}

Since each $f_k$ is supported on $|v|\leq 2$, and $\norm{\L\psi_k}_\infty \leq C_\psi$ by Lemma~\ref{thm:psi_properties}, property \eqref{L_of_psi}, we can apply the energy inequality from Lemma~\ref{thm:energy_inequality} to obtain
\[ \iint_{k+1} B(f_k,f_k) \leq C^k \iiint_k f_k^2 + C^k \iiint_k f_k + C^k \norm{f_k}_{L^{\sourceexp^\ast}(Q_k)}. \]
From this inequality, Lemma~\ref{thm:B_and_Hs}, and \eqref{indicator_bound}:
\[ \iiint_{k+1} \abs{\Lambda^s f_k}^2 \leq C^k \iiint_k f_{k-1}^2 + C^k \paren{\iiint_k f_{k-1}^2}^{1/\sourceexp^\ast}. \]
When $\norm{f_k}_2 < 1$, as we assume without loss of generality, the second term on the right-hand-side will dominate.  

Therefore, letting $\vexp$ be defined by $\frac{1}{\vexp} = \frac{1}{2} - \frac{s}{n}$, we have by Sobolev embedding
\begin{equation}\label{v_integration} 
\norm{f_k}_{L^{2,2,\vexp}(Q_{k+1} \times \R^n)} \leq C^k \norm{f_{k-1}}_{L^2(Q_k \times \R^n)}^{1/\sourceexp^\ast}. 
\end{equation}

Now we wish to utilize Riesz-Thorin interpolation to interpolate between this inequality and \eqref{t_and_x_integrability}.  

Consider $\theta \in [0,1]$ and the function
\[ \theta \mapsto \bracket{\frac{\theta}{2} + \frac{1-\theta}{\txexp}} - \bracket{\frac{\theta}{\vexp} + \frac{1-\theta}{1}}. \]
Because this function is negative at $\theta=0$ and positive at $\theta=1$, it must equal zero at some point $\theta^*$, and at this point we can define $q$ by
\[ \frac{1}{q} := \frac{\theta^*}{2} + \frac{1-\theta^*}{\txexp} = \frac{\theta^*}{\vexp} + \frac{1-\theta^*}{1}. \]
Moreover, since $1/q$ is a nontrivial convex combination of $1/2$ and $1/\txexp$, it must be the case that $q > 2$.  Riesz-Thorin tells us that
\[ \norm{f_k}_{L^{q,q,q}(Q_k \times \R^n)} \leq \norm{f_k}_{L^{2,2,\vexp}(Q_k \times \R^n)}^{\theta^\ast} \norm{f_k}_{L^{\txexp,\txexp,1}(Q_k \times \R^n)}^{1-\theta^\ast}. \]
Combining this with the bounds \eqref{t_and_x_integrability} and \eqref{v_integration},
\begin{equation}\label{recursive_inequality}
\norm{f_k}_{L^q(Q_k\times \R^n)} \leq C^k \norm{f_{k-2}}_{L^2(Q_{k-2}\times \R^n)}^{1-\frac{2}{\sourceexp} + \frac{\theta^\ast}{\sourceexp}}.
\end{equation}

This bound is the key to De Giorgi's first lemma.

\step{The recursion}

This step is standard to all De Giorgi arguments.  It does not depend on the specifics of our PDE \eqref{eq:main} in any way, except through the bound \eqref{recursive_inequality}.  

For any $k$, by \eqref{indicator_bound},
\begin{align*}
\iiint_k f_k^2 &= \iiint_k f_k^2 \indic{f_k>0}^{q-2}
\\ &\leq 2^{(k+1)(q-2)} \iiint_{k} f_{k}^2 f_{k-1}^{q-2}
\\ &\leq C^k \iiint_{k-1} f_{k-1}^q.
\end{align*}

From this and \eqref{recursive_inequality},
\[ \iiint_k f_k^2 \leq C^k \paren{\iiint_{k-3} f_{k-3}^2}^{\frac{q}{2}\paren{1 - \frac{2}{\sourceexp} + \frac{\theta^\ast}{\sourceexp}}}. \]
Since $q$ and $\theta^\ast$ are independent of $\sourceexp$ and $q > 2$, the exponent on this recursive inequality will be greater than 1 for $\sourceexp$ sufficiently large.  Specifically, the exponent exceeds 1 precisely when $\sourceexp > \sourceexp_0$, with $\sourceexp_0$ as defined in Section~\ref{sec:intro}, though we omit the explicit calculation.  

Since the exponent is greater than one, and the sequence
\begin{equation}\label{sequence} 
k \mapsto \iiint_k f_k^2 
\end{equation}
is monotone decreasing, by a standard fact about sequences (c.f. \cite{v}) we can now say that this sequence limits to 0 as $k \to \infty$, provided the initial value
\[ \iiint_{[-2,0]\times B_2 \times \R^n} \max(f-\psi^1,0)^2 \,dvdxdt \leq \delta_0 \]
is sufficiently small.  

Lastly, since the limit of that sequence \eqref{sequence} is zero, by the Lebesgue's monotone convergence theorem
\[ \iiint_{[-1,0] \times B_1 \times \R^n} (f - \psi^1 - \frac{1}{2})_+^2 \,dvdxdt = 0. \]
Since $\psi^1 = 0$ on $B_1$, the proposition is proven.  

\end{proof}


\section{Second De Giorgi Lemma}\label{sec:DG2}

In this section we will prove the second De Giorgi lemma, the intermediate value lemma.  It says that solutions to our PDE cannot have, in a small region, very much measure above a certain value and also very much measure below another value unless the solution also has sufficient measure between the two values.  The lemma is sometimes called an isoperimetric inequality. 

\begin{figure}[h]
\includegraphics[width=.25 \textwidth]{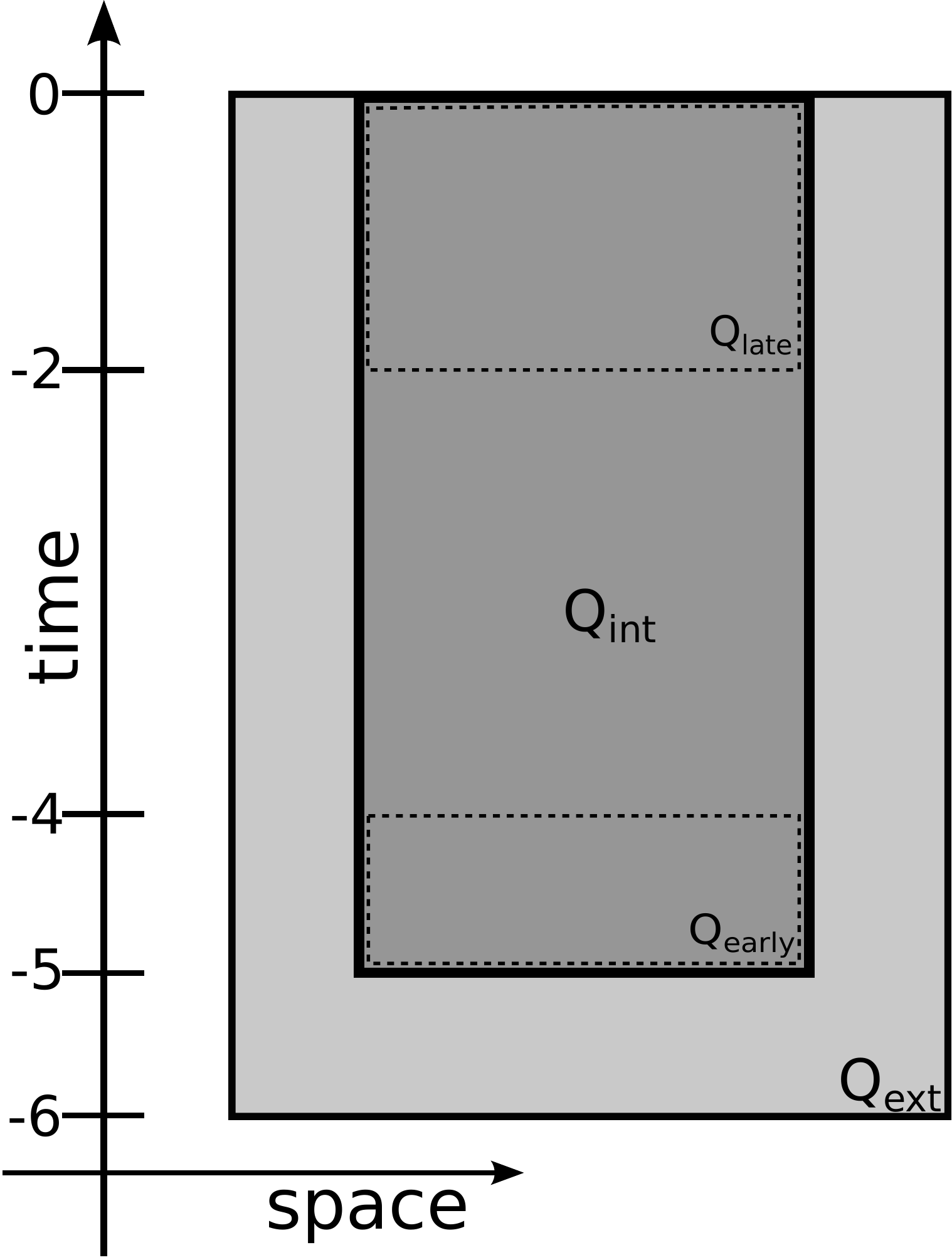}
\caption{Four overlapping cylinders described in Proposition~\ref{thm:DG2}.}
\end{figure}

To state Proposition~\ref{thm:DG2}, we must define four cylindrical regions in space-time:
\begin{align*} 
\Qext &:= [-6,0]\times B_3 \\
\Qint &:= [-5,0]\times B_2 \\
\Qearly &:= [-5,-4]\times B_2 \\
\Qlate &:= [-2,0]\times B_2. \\
\end{align*}

The constant $\delta_0$ in the statement of this proposition is defined in Proposition~\ref{thm:DG1}.  

\begin{proposition}[Second De Giorgi Lemma]\label{thm:DG2}
There exist universal constants $\gamma_0 > 0$ and $0 < \theta_0 < 1/3$ such that the following is true:

For any $f \in L^2(\Qext; H^s(\R^n))$ a weak solution to \eqref{eq:main} subject to \eqref{eq:kappa_bound} with
\[ \norm{a}_{L^\sourceexp(\Qext \times \R^n)} \leq \theta_0 \]
satisfying
\[ |f(t,x,v)| \leq 1 + \psi_{\theta_0}(v) \qquad \forall (t,x,v) \in \Qext \times \R^n, \]
if
\begin{equation}\label{DG2_mass_early} 
|\{f \leq 0\} \cap \Qearly \times B_2| \geq \frac{|\Qearly| \cdot |B_2|}{2} 
\end{equation}
and
\begin{equation}\label{DG2_mass_late}
|\{f \geq 1-\theta_0 \} \cap \Qlate \times B_2| \geq \delta_0 
\end{equation}
then
\begin{equation}\label{DG2_mass_between}
|\{0 < f < 1-\theta_0 \} \cap \Qint \times B_3| \geq \gamma_0. 
\end{equation}
\end{proposition}

As in other applications of De Giorgi's method, the idea of the proof is to produce a sequence of solutions to our PDE with smaller and smaller intermediate measure, show that they are compact and have a discontinuous limit, and then show that said limit function inherits enough regularity from the PDE to result in a contradiction.  

Our version of the proof is divided into four steps.  In the first step, we show that our sequence is uniformly differentiable in $v$.  We then use the averaging lemma to show that, in some very specific sense, our sequence is uniformly differentiable in $t$ and $x$.  In the second step, we combine the results of step one to obtain compactness in all three variables, thus producing our limit.  In the third step, we show that this limit function is regular in $v$.  The limit is constant in $v$ for $|v|$ small, and behaves like an indicator function depending only on $t$ and $x$.  In the fourth and final step, we show that certain $t$- and $x$-derivatives of our limit function are bounded, and that this contradicts what we know about the jump discontinuities in our limit.  

\begin{proof}
Assume that the theorem is false.  Then there must exist a sequence $f_i$ of solutions to our equation \eqref{eq:main} with operators $\L_i$ subject to \eqref{eq:kappa_bound} and source terms
\[ \norm{a_i}_{L^\sourceexp(\Qext \times \R^n)} \leq 1/i \]
such that 
\[ \abs{f_i(t,x,v)} \leq 1 + \psi_{1/i} \qquad \forall (t,x,v) \in \Qext \times \R^n \]
while
\begin{align*}
|\{f_i \leq 0\} \cap \Qearly \times B_2| &\geq \frac{|\Qearly|\cdot|B_2|}{2}, \\
|\{f_i \geq 1-\frac{1}{i} \} \cap \Qlate \times B_2| &\geq \delta_0, \\
|\{0 < f_i < 1-\frac{1}{i} \} \cap \Qint \times B_3| &\leq \frac{1}{i}.
\end{align*}
We wish to take a limit of these functions $f_i$.

\step{Regularity in $v$ and regularity in $t,x$}

Let $F:\R^n \to \R$ be a smooth radially-increasing function of $v$ which is identically $-1$ on $B_2$ and identically 0 outside of $B_3$.  Since $F \in \Ctest$, it is trivial to show that
\begin{equation}\label{LF_bounded} 
\norm{\L_i F}_\infty \leq C(n,s,\kappa). 
\end{equation}

\newcommand{\fip}{{f_i^+}}
\newcommand{\fim}{{f_i^-}}
To obtain compactness, we use a very blunt cutoff function $\bar{\psi}$ defined by
\begin{align*} 
\bar{\psi}(v) &:= \psi_\frac{1}{3}(v) + 1 + F(v), \\
\fip &:= \max\paren{f-\bar{\psi},0}, \\
\fim &:= \max\paren{\bar{\psi}-f,0}.
\end{align*}

Because $\psi_{1/3} \geq \psi_\theta$ for all $\theta < 1/3$ by Lemma~\ref{thm:psi_properties}, property \eqref{psi_increasing}, each $\fip$ for $i$ sufficiently large will be supported on $v \in B_3$.  In fact
\begin{equation}\label{f_leq_F} 
0 \leq \fip(t,x,v) \leq -F(v) \qquad \forall (t,x,v) \in \Qext \times \R^n. 
\end{equation}

Each $f_i$ is a solution to \eqref{eq:main}, so we can apply Lemma~\ref{thm:energy_inequality} on the regions $\Qext$ and $\Qint$ with cutoff $\bar{\psi}$.   From \eqref{LF_bounded} and Lemma~\ref{thm:psi_properties}, property \eqref{L_of_psi} we know that $\norm{\L_i\bar{\psi}}_\infty$ is bounded by a finite universal constant.   The right hand side of this energy inequality is then universally bounded by \eqref{f_leq_F} so
\begin{equation}\label{Bs_are_bounded}
\iint_\Qint B_i(\fip,\fip) \,dxdt - \iint_\Qint B_i(\fip,\fim) \,dxdt \leq C(n,s,\kappa). 
\end{equation}
In particular, by Lemma~\ref{thm:B_and_Hs},
\begin{equation}\label{v_regularity_of_fip} 
\iint_\Qint \int \abs{\Lambda^s \fip}^2 \,dvdxdt \leq C(n,s,\kappa). 
\end{equation}
Critically, the constant $C(n,s,\kappa)$ does not depend on $i$.  

Unfortunately the energy inequality does not give us regularity in the $t$ and $x$ variables. In order to obtain compactness, therefore, we must rely on an averaging lemma.  To that end, apply the transport operator to $\fip^2$ and obtain
\begin{align*} 
\kinet \fip^2 &= 2 \fip \kinet f_i
\\ &= 2 \fip \L_i f_i + 2 \fip a_i
\\ &= 2 \fip \L_i\paren{f_i - \bar{\psi}} + 2 \fip \L_i \bar{\psi} + 2\fip a_i.
\end{align*}

For any function $g$ and operator $\L$ satisfying \eqref{eq:kappa_bound}, and $g_+ := \max(g,0)$, it is true that, for any $t$, $x$ fixed,
\begin{align*}
2 g_+ \L g &= \int 2 [g_+(v) g(w) - g_+(v)^2] K(t,x,v,w) \,dw
\\ &= \int [g_+(w)^2 - g_+(v)^2] K \,dw + \int [2 g_+(v) g(w) - g_+(v)^2 -  g_+(w)^2] K \,dw
\\ &= \int [g_+(w)^2 - g_+(v)^2] K \,dw - \int [g_+(w) - g_+(v)]^2 K \,dw + \int 2 g_+(v) [g(w) - g_+(w)] K \,dw
\\ &= \L g_+^2 - \int [g_+(w) - g_+(v)]^2 K \,dw - 2 \int g_+(v)g_-(w) K \,dw.
\end{align*}

Thus
\[ \kinet \fip^2 = H := H_1 + H_2 + H_3 + H_4 \]
where
\begin{align*} 
H_1 &:= \L_i \paren{\fip^2},
\\ H_2 &:= - \int [\fip(w) - \fip(v)]^2 K(v,w) \,dw,
\\ H_3 &:= - 2 \int \fip(v)\fim(w) K(v,w) \,dw,
\\ H_4 &:=  2 \fip \L_i \bar{\psi} + 2 \fip a_i.
\end{align*}

We proceed to bound $H$, term by term, independent of $i$.  

We begin with an $H^s$ bound on $\fip^2$:
\begin{align}
\int \abs{\Lambda^s (\fip^2)}^2 \,dv &= \iint \frac{|\fip^2(w)-\fip^2(v)|^2}{|v-w|^{n+2s}}\,dwdv \nonumber
\\ &= \iint \bracket{\fip(w)+\fip(v)}^2 \frac{|\fip(w)-\fip(v)|^2}{|v-w|^{n+2s}}\,dwdv \nonumber
\\ &\leq 2^2 \norm{\fip}_{L^\infty}^2 \int \abs{\Lambda^s (\fip)}^2 \,dv. \label{fip_in_Hs}
\end{align} 
From this, the bounds \eqref{f_leq_F} and \eqref{v_regularity_of_fip}, and Lemma~\ref{thm:B_and_Hs}, we obtain
\begin{equation}\label{bound_on_H1}
\norm{\bessel^{-s/2} H_1}_{L^2(\Qint\times\R^n)} \leq C(n,s,\kappa). 
\end{equation}

The terms $H_2$ and $H_3$ are strictly negative, so their total variations as measures are simply the absolute values of their integrals.  Thus their norms in $\mathcal{M}(\Qint \times \R^n)$ are
\[ \abs{\iint_{\Qint} \int H_2 \,dvdxdt} = \iint_{\Qint} B_i(\fip,\fip) \,dxdt, \]
\[ \abs{\iint_{\Qint} \int H_3 \,dvdxdt} = -\iint_{\Qint} B_i(\fip,\fim) \,dxdt. \]
These are of course universally bounded by \eqref{Bs_are_bounded}.  

Recall that $\paren{1-\Laplace_{v}}^{-\paren{s + \frac{n}{2}}/2}$ can be represented as convolution with a Green's function $G_{s+n/2}(v)$ (see e.g. Stein \cite{s}).  The function $G_{s+n/2}$ decays exponentially as $|v|\to \infty$ and has a singularity like $\frac{1}{|v|^{\frac{n}{2} - s}}$ near zero.  Therefore $G_{s+n/2}$ is in $L^2$.  By Young's Inequality, convolution of a measure and an $L^2$ function is bounded by the product of their $\mathcal{M}$ and $L^2$ norms respectively, so
\begin{equation}\label{bound_on_H2} 
\norm{\paren{1-\Laplace_v}^{-\paren{s + \frac{n}{2}}/2}H_2}_{L^2(\Qint \times \R^n)} \leq C(n,s,\kappa), 
\end{equation}
\begin{equation}\label{bound_on_H3}
\norm{\paren{1-\Laplace_v}^{-\paren{s + \frac{n}{2}}/2}H_3}_{L^2(\Qint \times \R^n)} \leq C(n,s,\kappa). 
\end{equation}

Lastly, from \eqref{f_leq_F} and since $\sourceexp \geq 2$ we know
\begin{equation}\label{bound_on_H4}
\norm{H_4}_{L^2(\Qint \times \R^n} \leq C(n,s,\kappa). 
\end{equation}

Finally we are ready to apply Lemma~\ref{thm:avg_lemma} to $\fip^2$, which says for any $\eta \in \Ctest(\R^n)$ and any subset $\bar{\Omega}$ compactly contained in the interior of $\Qext$,
\[ \norm{\int \eta \fip^2 \,dv}_{H^\alpha(\bar{\Omega})} \leq C(\eta, \bar{\Omega}) \paren{\norm{\fip^2}_{L^2(\Qint \times \R^n)} + \norm{\bessel^{-\paren{s + \frac{n}{2}}/2} H}_{L^2(\Qint\times\R^n)}} \]
where
\[ \alpha = \paren{2\paren{s+\frac{n}{2}}}\n. \]

From \eqref{bound_on_H1}, \eqref{bound_on_H2}, \eqref{bound_on_H3}, and \eqref{bound_on_H4}, we can say that in fact
\begin{equation}\label{Halpha_bounded} 
\norm{\int \eta \fip^2 \,dv}_{H^\alpha(\bar{\Omega})} \leq C(n,s,\kappa,\eta,\bar{\Omega}). 
\end{equation}

\step{Producing a strong $L^2$ limit}

Since all the $\fip$ are bounded by \eqref{f_leq_F}, $\{\fip^2\}_i$ is a bounded subset of $L^2(\Qint \times \R^n)$.  By Banach-Alaoglu, there exists a function $f^+$ such that, along some subsequence, 
\[ \fip^2 \weakly {f^+}^2 \]
weakly in $L^2(\Qint \times \R^n)$.  

Our goal is to show that this limit converges also strongly in $L^2_\loc(\Qint; L^2(\R^n))$.  To that end, fix some compact subset $\bar{\Omega}$ of $\Qint$.  

Strong and weak limits, when both exist, must be equal, so with the bound \eqref{Halpha_bounded} we apply Rellich-Kondrachov to prove that
\[ \int \eta(v) \fip^2 \,dv \to \int \eta(v) {f^+}^2 \,dv \]
strongly in $L^2(\bar{\Omega})$, without passing to a further subsequence, for any $\eta \in \Ctest(\R^n)$.  

In particular, if we fix some $\eta$ such that $\eta_\eps(v) = \eps^{-n} \eta(v/\eps)$ is an approximation to the identity, then for $\eps > 0$ and $v \in \R^n$ fixed,
\[ \iint_{\bar{\Omega}} \bracket{\int \fip^2(w) \eta_\eps(v-w) \,dw - \int {f^+}(w)^2 \eta_\eps(v-w) \,dw}^2 \,dxdt \xrightarrow{i\to\infty} 0. \]
Note that this is pointwise (in $v$) convergence of convolutions.

Since the $\fip$ are all bounded by \eqref{f_leq_F}, and by weak convergence so is ${f^+}$, we can apply the Lebesgue dominated convergence theorem to conclude that not only do these convolutions converge pointwise in $v$, but they converge in integral as well.  That is,
\begin{equation}\label{convergence_of_convolutions}
\int \iint_{\bar{\Omega}} \bracket{\paren{\fip^2 \ast_v \eta_\eps}(v) - \paren{{f^+}^2 \ast_v \eta_\eps}(v)}^2 \,dxdtdv \to 0. 
\end{equation}

It is known (see Lemma~\ref{thm:convolution_estimate} in the appendix for a proof) that for any $g \in H^s(\R^n)$, 
\[ \norm{g - g\ast \eta_\eps}_{L^2(\R^n)} \leq C(n,s,\eta) \norm{g}_{H^s(\R^n)} \eps^s. \]

Therefore for our functions $\fip^2$, 
\[ \iint_{\bar{\Omega}} \int \paren{\fip^2 - \fip^s \ast_v \eta_\eps}^2 \,dvdxdt \leq C(n,s,\eta) \eps^{2s} \iint_\Qint \int |\Lambda^s \fip|^2 \,dvdxdt. \]

Remember that $\norm{\fip^2}_{L^2(\Qint;H^s(\R^n))}$ is bounded by \eqref{fip_in_Hs} and \eqref{v_regularity_of_fip}, and, since the $H^s$ norm is weakly lower-semi-continuous, $\norm{{f^+}^2}_{L^2(\Qint;H^s(\R^n))}$ will be bounded as well.  

Therefore we can bound
\begin{align*} 
\norm{\fip^2 - f_+^2}_2 &\leq \norm{\fip^2 - \eta_\eps \ast_v \fip^2}_2 + \norm{\eta_\eps \ast_v \fip^2 - \eta_\eps \ast_v {f^+}^2}_2 + \norm{{f^+}^2 - \eta_\eps \ast_v {f^+}^2}_2
\\ &\leq C \eps^s + \norm{\eta_\eps \ast_v \fip^2 - \eta_\eps \ast_v {f^+}^2}_2.
\end{align*}
By $\norm{\cdot}_2$ we mean $\norm{\cdot}_{L^2(\bar{\Omega}\times\R^n)}$.  For any $\delta > 0$, we take $\eps$ small enough that $C \eps^s \leq \delta/2$.  Then with $\eps$ fixed, we choose $i$ large enough that (by \eqref{convergence_of_convolutions}) $\norm{\eta_\eps \ast_v \fip^2 - \eta_\eps \ast_v {f^+}^2}_2 \leq \delta/2$.  This proves that $\norm{\fip^2 - f_+^2}_2$ goes to 0 as $i \to \infty$.  

Since this is true for any $\bar{\Omega}$ compactly contained in the interior of $\Qint$, we can say that $\fip^2 \to {f^+}^2$ in $L^2_\loc(\Qint;L^2(\R^n))$.  

In fact, since our domain is compact, this convergence happpens in $L^1_\loc(\Qint;L^2(\R^n))$ as well.  Since $\fip$ and $f_+$ are non-negative, and since $(x-y)^2 \leq \abs{x^2 - y^2}$ for any non-negative real numbers $x$ and $y$, we can say that
\[ \fip \to {f^+} \qquad \textrm{in } L^2_\loc(\Qint; L^2(\R^n)). \]

\step{The limit is constant in $v$}

We'll denote 
\[ f = f^+ + 1 + F. \]

Because $\fip \to f^+$ strongly in $L^2_\loc$, we know that
\begin{equation}\label{mass_of_limit}
\begin{aligned}
|\{f = 0\} \cap \Qearly \times B_2| &\geq \frac{|\Qearly|\cdot|B_2|}{2}, \\
|\{f = 1 \} \cap \Qlate \times B_2| &\geq \delta_0, \\
|\{1+F < f < 1 \} \cap \Qint \times B_3| &= 0. 
\end{aligned}
\end{equation}

\begin{remark}
If $s \geq 1/2$, we can use the fact that the $H^s_v$ norm of $f$ is known to be finite for almost every $t,x$ fixed and obtain \eqref{constant_in_v} immediately, making the remainder of Step 3 unnecessary.  It is only in the case $s < 1/2$ that this regularity in $v$ is insufficient to rule out jump discontinuities.  Therefore we follow the technique used in \cite{nio} and by Bass and Kassmann in \cite{bk} to exploit the energy inequality's cross term.  
\end{remark}

\newcommand{\flambda}{f_{i,\lambda}^+}
\newcommand{\flambdam}{f_{i,\lambda}^-}
For $0 < \lambda \ll 1$, define the functions
\begin{align*} 
\flambda &:= \paren{f_i - \psi_\lambda - 1 - \lambda F}_+, \\
\flambdam &:= \paren{f_i - \psi_\lambda - 1 - \lambda F}_-.
\end{align*}

From the the energy inequality of Lemma~\ref{thm:energy_inequality}, we see that for all $i$ the cross term is bounded
\begin{equation}\label{flambda_energy_inequality}
-\iint_\Qint B\paren{\flambda, \flambdam} \leq C(n,s,\kappa) \bracket{\iint_\Qext \int {\flambda}^2 + \sup_{v\in B_3} \L_i(\psi_\lambda + \lambda F) \iint_\Qext\int \flambda + \norm{a_i}_{\sourceexp} \norm{\flambda}_{\sourceexp^\ast}}. 
\end{equation}

For $v \in B_3$, Lemma~\ref{thm:psi_properties}, property \eqref{L_of_psi} says that $\L_i \psi_\lambda(v) \leq C_\psi \lambda^{3s/2}$.  Moreover by \eqref{LF_bounded}, $\abs{\L_i \lambda F(v)} \leq C\lambda$ for some universal constant $C$.  

For $\lambda$ fixed and $i$ sufficiently large, 
\[ f_i \leq 1+\psi_{1/i} \leq 1+\psi_\lambda \]
so
\[ 0 \leq \flambda \leq \lambda F. \]
Therefore, for $\lambda$ fixed and $i$ sufficiently large, the inequality \eqref{flambda_energy_inequality} yields
\[ \iint_\Qint -B\paren{\flambda, \flambdam} \leq C(n,s,\kappa) \bracket{\lambda^2 + (\lambda + \lambda^{3s/2})\lambda + (1/i)\lambda}. \]

The cross term in turn bounds the integral of $\flambda$ and $\flambdam$.  For any $t,x$ fixed
\begin{align*}
-B_i(\flambda, \flambdam) &= \iint K(v,w) \flambda(v) \flambdam(w) \,dwdv
\\ &\geq \frac{1}{\kappa} \iint_{|v-w| \leq 6} \frac{\flambda(v) \flambdam(w)}{|v-w|^{n+2s}} \,dwdv
\\ &\geq \frac{1}{\kappa} \iint_{|v|\leq 3, |w|\leq 3} \frac{\flambda(v) \flambdam(w)}{6^{n+2s}} \,dwdv
\\ &= C \int_{B_3} \flambda \,dv \int_{B_3} \flambdam \,dv.
\end{align*}

Since $f_i \to f$ strongly in $L^2_\loc(\Qint;L^2(\R^n))$, these upper- and lower-bounds on the cross term hold in the limit:
\begin{equation}\label{bound_on_good_term}
\iint_\Qint \bracket{ \int_{B_3} \paren{f-\psi_\lambda-1-\lambda F}_+ \,dv \int_{B_3} \paren{f-\psi_\lambda-1-\lambda F}_- \,dv} \,dxdt \leq C(n,s,\kappa) (\lambda^2 + \lambda^{1+3s/2}). 
\end{equation}

This bound on the limit $f$ is very strong, because by \eqref{mass_of_limit} we have either $f(t,x,v)=1$ or $f(t,x,v) = 1 + F(v)$ for almost all $(t,x,v) \in \Qint \times B_3$.  For such $(t,x,v)$, also $\psi_\lambda(v)=0$ and so
\[ f-\psi_\lambda - 1 - \lambda F = \bracket{-\lambda F} \indic{f=1	} + \bracket{(1-\lambda) F} \indic{f=1+F}. \]

The function $-\lambda F$ is non-negative, while $(1-\lambda)F$ is non-positive, so at any point $t,x \in \Qint$,
\begin{align*}
\int_{B_3} \paren{f-\psi_\lambda-1-\lambda F}_+ \,dv &= -\lambda \int F \indic{f=1} \,dv \\
\int_{B_3} \paren{f-\psi_\lambda-1-\lambda F}_- \,dv &= -(1-\lambda) \int F \indic{f=1+F} \,dv.
\end{align*} 

Plugging this into \eqref{bound_on_good_term} and moving all the $\lambda$ to one side, we obtain
\[ \iint_\Qint \int F \indic{f=1} \,dv \int F \indic{f=1+F} \,dv \,dxdt \leq C(n,s,\kappa) \frac{\lambda^2 + \lambda^{1+3s/2}}{\lambda(1-\lambda)}. \]
The left-hand side is independent of $\lambda$, and the right side tends to 0 as $\lambda \to 0$, so we conclude that the left-hand side is in fact 0.  In particular, this means that for almost every $t,x \in \Qint$, either
\begin{equation}\label{constant_in_v}
\abs{\{v: f(t,x,v) = 1\} \cap B_3 } = 0 \qquad \textrm{or} \qquad \abs{\{v: f(t,x,v) = 1+F\} \cap B_3 } = 0. 
\end{equation}

\step{The limit has bounded derivative, which is a contradiction}

What remains is to argue that $f$ increases from 0 to 1, without taking intermediate values along the way, despite having bounded derivative.  Moreover, it is not enough to bound the derivatives in any weak sense, because jump discontinuities can hide in sets of measure zero.  

Since $f$ is only defined up to an a.e.-equivalence class, we can assume without loss of generality that, for every (not a.e.) $t,x \in \Qint$, either $f(t,x,v) \equiv 1$ or $f(t,x,v) \equiv 1+F$.  

For each $i$, since $\bar{\psi}$ is constant in $t$ and $x$, it is true that
\[ \kinet\paren{f_i - \bar{\psi}} = \L_i \paren{f_i - \bar{\psi}} + \L_i\bar{\psi} + a_i. \]
Multiplying by $\indic{f_i \geq \bar{\psi}}$ and recalling the standard pointwise inequality for integral operators (c.f. \cite{cs}),
\[ \kinet \fip \leq \L_i \fip + \indic{f_i \geq \bar{\psi}} \L_i\bar{\psi} + \indic{f_i \geq \bar{\psi}} a_i. \]

By \eqref{LF_bounded} and Lemma~\ref{thm:psi_properties}, property \eqref{L_of_psi}, the term $\indic{f_i \geq \bar{\psi}} \L_i\bar{\psi}$ is less than a universal constant $C(n,s,\kappa)$, and of course the $L^\sourceexp$ norm of $\indic{f_i \geq \bar{\psi}} a_i$ is less than $1/i$ so this term will vanish in the limit.  Let $\phi \in \Ctest(\Qint \times \R^n)$ be a non-negative test function and consider
\[ -\chevron{\fip, \kinet\phi} \leq \chevron{\fip, \L_i \phi} + \chevron{C, \phi} + \frac{1}{i} \norm{\phi}_{\sourceexp^\ast}. \]

For $\phi \in \Ctest$ fixed, the functions $\L_i \phi$ will be uniformly bounded in $L^\infty$ and decay like $|v|^{-n-2s}$.  In particular they are uniformly bounded in $L^2(\Qint \times \R^n)$.  Therefore
\[ \chevron{\fip - f^+, \L_i \phi} \to 0 \]
so in little-o notation
\[ -\chevron{\fip, \kinet\phi} \leq \chevron{\L_i f^+, \phi} + \chevron{C, \phi} + o(1). \]



By \eqref{constant_in_v} and \eqref{LF_bounded},
\[ \L_i f^+ = - \indic{t,x:f\equiv 1} \L_i F \leq C(n,s,\kappa). \]
Thus for some universal constant $C_1 = C_1(n,s,\kappa)$ we have, in the sense of distributions,
\[ \kinet \paren{f - 1 - F} \leq C_1. \]


To make the remaining calculation rigorous, let $\eta_\eps(t,x)$ be an approximation to the identity and define 
\[ f_\eps = \eta_\eps \ast_{t,x} f. \]
These functions $f_\eps$ are smooth and $f_\eps \to f$ pointwise a.e. as $\eps \to 0$.  For $(t,x)\in \Qint$ fixed, $f_\eps$, like $f$, is constant over all $v \in B_2$.  Because the transport operator commutes with convolution in $t$ and $x$,
\[ \kinet f_\eps = \eta_\eps \ast_{t,x} \kinet f \leq C_1. \]
This inequality is true not only in the sense of distributions but also pointwise because the functions are smooth.  

Define two sets
\begin{align*} 
M_1 &= \{t,x \in \Qlate: f(t,x,v) = 1 \}, \\
M_0 &= \{t,x \in \Qearly: f(t,x,v) = 1+F(v) \}. \\
\end{align*}
By \eqref{mass_of_limit} we know that $|M_0| \geq \frac{|\Qearly|}{2}$ and $|M_1| \geq \frac{\delta_0}{|B_2|}$.  By Egorov's theorem, for $\eps$ sufficiently small,
\begin{align}
|M_1^\eps| &:= \abs{\{t,x \in \Qlate: f_\eps(t,x,v) > 0.9 \, \forall v \in B_2\}} \geq \frac{\delta_0}{2|B_2|}, \label{size_of_M1eps} \\
|M_0^\eps| &:= \abs{\{t,x \in \Qearly: f_\eps(t,x,v) < 0.1 \, \forall v \in B_2\}} \geq \frac{|\Qearly|}{4}. \nonumber
\end{align}
Fixing $\eps$, choose a point $(t_0,x_0) \in M_0^\eps$.


For any $(t_1,x_1) \in M_1^\eps$, we can define the velocity $\bar{v} := \frac{x_1-x_0}{t_1-t_0}$ and see that $|\bar{v}| \leq 2$.  Then the function
\[ \tau \mapsto f_\eps\big((1-\tau)t_0 + \tau t_1, (1-\tau)x_0 + \tau x_1, \bar{v}\big) \]
is equal to 0 at $\tau = 0$ and equal to 1 at $\tau = 1$, and its derivative is less than $(t_1-t_0) C_1$.  Therefore
\begin{equation}\label{line_segments_have_large_measure}
\mathcal{H}^1\paren{\textrm{segment}\bracket{(t_0,x_0),(t_1,x_1)} \cap \{t,x: 0.1 < f(t,x,v) < 0.9 \, \forall v\in B_2\}} \geq \frac{.8 \sqrt{1 + |\bar{v}|^2}}{C_1} \geq \frac{2}{C_1}. 
\end{equation}

The facts \eqref{size_of_M1eps} and \eqref{line_segments_have_large_measure} tell us, by the elementary geometric argument of Lemma~\ref{thm:cone}, that the cone with vertex $(t_0,x_0)$ and base $M_1^\eps$ must intersect $\{t,x: 0.1 < f(t,x,v) < 0.9 \, \forall v\in B_2\}$ on a set with measure $(\delta_0/2|B_2|)(2/C_1)^2/80$.  

In particular,
\[ \abs{\{0.1<f_\eps<0.9\} \cap \Qint \times B_2} \geq \frac{2\delta_0}{80 C_1^2|B_2|} > 0. \]

This bound holds for all $\eps$ sufficiently small, but we know from \eqref{mass_of_limit} that it is not true for $f$.  By Egorov's theorem, this is a contradiction.  

Therefore our sequence $f_i$ must not exist, and the proposition must be true.  

\end{proof}


\section{H\"{o}lder Continuity}\label{sec:Holder}

In this section, we explain how Proposition~\ref{thm:DG1} and Proposition~\ref{thm:DG2} together lead to H\"{o}lder regularity of our solution.  We begin by showing that the PDE \eqref{eq:main} is scaling invariant.  We then show, in Lemma~\ref{thm:oscillation}, how to combine Proposition~\ref{thm:DG1} and Proposition~\ref{thm:DG2} to create a sort of Harnack's inequality.  The ideas here are not new, in particular we follow \cite{nio} very closely.  

\begin{lemma}[Scaling]\label{thm:scaling}
If $f$ solves \eqref{eq:main} on some region $Q\times \R^n \subseteq \Rall$, then for any constant $\eps < 1$, 
\[ \bar{f}(t,x,v) := f(\eps^{2s} t, \eps^{1+2s} x, \eps v) \]
will solve 
\[ \del_t \bar{f} + v \cdot \grad_x \bar{f} = \int [\bar{f}(w) - \bar{f}(v)] \bar{K}(t,x,v,w) \,dw + \bar{a} \]
on the appropriate region $Q_\eps\times\R^n$ with $\bar{K}$ symmetric and satisfying \eqref{eq:kappa_bound}, and with 
\[ \norm{\bar{a}}_{L^\sourceexp(Q_\eps\times\R^n)} \leq \eps^{2s\paren{1 - \frac{n+1+n/s}{\sourceexp}}} \norm{a}_{L^\sourceexp(Q\times\R^n)}. \]
\end{lemma}

\begin{proof}
Denote
\[ p = (t,x,v), \qquad \bar{p} = (\eps^{2s} t, \eps^{1+2s} x, \eps v). \]

Evaluate the equality \eqref{eq:main} at the point $\bar{p}$, so that
\begin{equation}\label{PDE_at_barp}
(\del_t f)(\bar{p}) + \eps v \cdot (\grad_x f)(\bar{p}) = (\L f)(\bar{p}) + a(\bar{p}). 
\end{equation}

For our modified function $\bar{f}$ evaluated at $p$,
\begin{align} 
\del_t \bar{f}(p) &= \eps^{2s} (\del_t f)(\bar{p}), \label{scaled_delt} \\
\grad_x \bar{f}(p) &= \eps^{1+2s} (\grad_x f)(\bar{p}). \label{scaled_gradx}
\end{align}

Define
\[ \bar{K}(t,x,v,w) := \eps^{n+2s} K(\eps^{2s} t, \eps^{1+2s} x, \eps v, \eps w). \]
It's clear that $\bar{K}$ is still symmetric.  Since
\[ \bar{K}(t,x,v,w) \geq \eps^{n+2s} \indic{\eps |v-w| \leq 6} \frac{1}{\kappa} (\eps |v-w|)^{-(n+2s)} \geq \indic{|v-w| \leq 6} \frac{1}{\kappa} |v-w|^{-(n+2s)} \]
and 
\[ \bar{K}(t,x,v,w) \leq \eps^{n+2s} \kappa (\eps |v-w|)^{-(n+2s)} = \kappa |v-w|^{-(n+2s)}, \]
$\bar{K}$ satisfies the bound \eqref{eq:kappa_bound}.  

For this $\bar{K}$, 
\begin{align} 
\int [\bar{f}(w) - \bar{f}(v)] \bar{K}(p,w) \,dw &= \eps^{n+2s} \int [f(\eps w) - f(\eps v)] K(\bar{p}, \eps w) \,dw \nonumber
\\ &= \eps^{n+2s} \frac{1}{\eps^n} \int [f(\eps w) - f(\eps v)] K(\bar{p}, \eps w) \,d(\eps w) \nonumber
\\ &= \eps^{2s} (\L f)(\bar{p}). \label{scaled_L}
\end{align}

Define
\begin{equation}\label{scaled_sigma}
\bar{a}(t,x,v) := \eps^{2s} a(\eps^{2s} t, \eps^{1+2s} x, \eps v). 
\end{equation}
Then the $L^\sourceexp$ norm of $\bar{a}$ is
\[ \norm{\bar{a}}_\sourceexp = \eps^{2s} \eps^{-\frac{2s + n(1+2s) + n}{\sourceexp}} \paren{\iiint a(\eps^{2s} t, \eps^{1+2s} x, \eps v)^\sourceexp \,d(\eps^{2s} t)\, d(\eps^{1+2s} x) \, d(\eps v)}^{1/\sourceexp}. \]

Plugging \eqref{scaled_delt}, \eqref{scaled_gradx}, \eqref{scaled_L}, and \eqref{scaled_sigma} into \eqref{PDE_at_barp} yields
\[ \eps^{-2s} \del_t \bar{f}(p) + \eps \eps^{-1-2s} v\cdot\grad_x \bar{f}(p) = \eps^{-2s} \int [\bar{f}(w)-\bar{f}(v)] \bar{K}(p) \,dw + \eps^{-2s} \bar{a}(p). \]
Multiply both sides by $\eps^{2s}$ to obtain our desired result.  
\end{proof}

\begin{remark}
In addition to scaling, we can also translate solutions of \eqref{eq:main}.  If $f$ is a solution and $(t_0,x_0,v_0)$ is a point in its domain, then
\[ \bar{f}(t,x,v) := f(t_0 + t, x_0 + x + v_0 t, v_0 + v) \]
will be a solution to \eqref{eq:main} with similarly adjusted source term and kernel.  This translation invariance is necessary for the proof of H\"{o}lder continuity, though we omit any further detail.  
\end{remark}

The following lemma should be thought of as a Harnack inequality, except that it keeps track also of the growth in $v$.  

In the sequel, $\theta_0$ and $\gamma_0$ refer to the constant defined in the statement of Proposition~\ref{thm:DG2}, and $\delta_0$ refers to the constant defined in Proposition~\ref{thm:DG1} which is used again in the statement of Proposition~\ref{thm:DG2}.

\begin{lemma}[Oscillation Lemma]\label{thm:oscillation}
There exists a universal constant $0 < \lambda < 1$ such that the following is true: 

If $f \in L^2(\Qext; H^s(\R^n))$ is a weak solution to \eqref{eq:main} subject to \eqref{eq:kappa_bound} with source term
\[ \norm{a}_{L^\sourceexp(\Qext)} \leq \lambda \theta_0 \]
and satisfying
\begin{equation}\label{f_leq_1+psi}
|f(t,x,v)| \leq 1 + \lambda \psi_{\theta_0}(v) 
\end{equation}
for all $t,x \in \Qext$, $v \in \R^n$, then
\[ \bracket{\sup_{[-1,0]\times B_1 \times B_1} f} - \bracket{\inf_{[-1,0]\times B_1 \times B_1} f} \leq 2 - \lambda. \]

Moreover, at least one of the two functions
\[ \bar{f}_1(t,x,v) = \paren{1+\frac{\lambda}{2}}\bracket{f(\lambda^{2s} t, \lambda^{1+2s} x, \lambda v) + \lambda/2} \]
or
\[ \bar{f}_2(t,x,v) = \paren{1+\frac{\lambda}{2}}\bracket{f(\lambda^{2s} t, \lambda^{1+2s} x, \lambda v) - \lambda/2} \]
will also solve \eqref{eq:main} subject to \eqref{eq:kappa_bound} in the weak sense with source term smaller than $\lambda \theta_0$ and satisfy 
\[ |\bar{f}_i(t,x,v)| \leq 1 + \lambda \psi_{\theta_0}(v) \]
for all $t,x \in \Qext$, $v \in \R^n$.  

\end{lemma}

\begin{proof}
Choose $k_0 \in \N$ such that
\[ \gamma_0 k_0 > |\Qint \times B_3|. \]
Take $\lambda$ small enough that
\begin{equation}\label{lambda_assumptions}
\lambda \leq \frac{\theta_0^{k_0+1}}{2}, \qquad 3\lambda^{1+2s} < 1, \qquad 6\lambda^{2s} < 1, \qquad \lambda < \eps_0, \qquad \textrm{ and } \paren{1+\frac{\lambda}{2}}\lambda^{2s\paren{1 - \frac{n+1+n/s}{\sourceexp}}} \leq 1
\end{equation}
where $\eps_0 = \eps_0(s,\theta_0)$ is defined in Lemma~\ref{thm:psi_properties} property \eqref{psi_scaled_inequality}.  

Assume without loss of generality that 
\begin{equation} \label{fk_can_haz_measure_below}
|\{f\leq 0\} \cap \Qearly \times B_2| \geq |\Qearly|\cdot |B_2|/2. 
\end{equation}  
If this were not true, then we could simply discuss $-f$ instead.  This proposition holds for $f$ if and only if it holds for $-f$.  

With this assumption, we will assert that the proposition's result is true for
\[ \bar{f}(t,x,v) = \paren{1+\frac{\lambda}{2}}\bracket{f(\lambda^{2s} t, \lambda^{1+2s} x, \lambda v) + \lambda/2}. \]
It is clear by Lemma~\ref{thm:scaling} and linearity of Equation~\eqref{eq:main} that $\bar{f}$ will solve \eqref{eq:main} subject to \eqref{eq:kappa_bound} with source term $\bar{a}$ smaller than $\lambda \theta_0$ by \eqref{lambda_assumptions}.  We must show that $\bar{f}$ is also bounded as desired.  

Consider the sequence of functions
\begin{align*}
f_0 &= f \\
f_k &= \frac{f_{k-1}-1}{{\theta_0}} + 1 = \frac{f-1}{{\theta_0}^k} + 1.
\end{align*}
Since equation~\eqref{eq:main} is linear, all $f_k$ will also be solutions with source terms $\frac{1}{\theta_0^k} a$.  

For each $0 \leq k \leq k_0+1$ and any $(t,x,v) \in \Qext\times\R^n$, 
\[ |a(t,x,v)| \leq \frac{\lambda\theta_0}{\theta_0^k} \leq \theta_0 \]
by the assumption \eqref{lambda_assumptions}, and by \eqref{f_leq_1+psi} and \eqref{lambda_assumptions},
\begin{equation}\label{fk_has_bounded_growth} 
f_k = \frac{f-1}{{\theta_0}^k} + 1 \leq \frac{\lambda}{{\theta_0}^k} \psi_{\theta_0} + 1 \leq \psi_{\theta_0} + 1. 
\end{equation}

We wish to show that $f_{k_0}$ satisfies
\begin{equation}\label{fk_can_haz_DG1}
|\{f_{k_0} \geq 1-{\theta_0}\}\cap \Qlate \times B_2| \leq \delta_0. 
\end{equation}
Therefore assume, for contradiction, that \eqref{fk_can_haz_DG1} does not hold.  Then by construction, each $f_k$ will satisfy \eqref{DG2_mass_late} for $0 < k \leq k_0$.  Moreover, all $f_k$ will satisfy \eqref{DG2_mass_early} since $f_0$ does by \eqref{fk_can_haz_measure_below}.  Therefore we can apply Proposition~\ref{thm:DG2} and conclude that each $f_k$ for $k$ from 0 to $k_0$ must satisfy \eqref{DG2_mass_between}.  That means that the set
\[ S_k := |\{f_k \leq 0\} \cap \Qint \times B_3| \]
must increase in measure by at least $\gamma_0$ with each increment of $k$.  By choice of $k_0$, this would be a contradiction.  We conclude that \eqref{fk_can_haz_DG1} holds.  

Due to \eqref{fk_has_bounded_growth} and Lemma~\ref{thm:psi_properties}, property \eqref{psi_geq_one_plus_psi}, we say that for all $t,x \in \Qlate$ and all $|v| \geq 2$
\[ f_{k_0+1}(t,x,v) \leq 1 + \psi_{\theta_0}(v) \leq \psi^1(v). \]
By \eqref{fk_has_bounded_growth}, $f_{k_0+1}(t,x,v) \leq 1$ for all $(t,x,v) \in [-2,0]\times B_2 \times B_2$, so we can say by \eqref{fk_can_haz_DG1} that
\[ \iiint_{\Qlate \times B_2} \max(f_{k_0+1}-\psi^1,0)^2 \,dvdxdt \leq \delta_0. \]

This is sufficent to apply Proposition~\ref{thm:DG1} to $f_{k_0+1}$ and conclude that $ f_{k_0+1} \leq 1/2$ on $[-1,0]\times B_1 \times B_1$.  Thus for the original $f$,
\begin{equation}\label{Harnack_partial_result} 
-1 \leq f \leq 1-\frac{1}{2} {\theta_0}^{k_0+1} \leq 1-\lambda \qquad \forall (t,x,v)\in [-1,0]\times B_1 \times B_1. 
\end{equation}
This proves the lemma's first claim.  

We now know from \eqref{Harnack_partial_result}, the definition of $\bar{f}$, and \eqref{lambda_assumptions} that for all $t,x \in \Qext$ and $|v| \leq \lambda\n$
\begin{align*}
\bar{f}(t,x,v) &\leq \paren{1+\frac{\lambda}{2}} \bracket{ 1 - \lambda + \lambda/2} \leq 1, \\
\bar{f}(t,x,v) &\geq \paren{1+\frac{\lambda}{2}} \bracket{-1 + \lambda/2} \geq -1.
\end{align*}

For $t,x \in \Qext$ and $|v| \geq \lambda\n$, since $\lambda < \eps_0$, we know by Lemma~\ref{thm:psi_properties}, property \eqref{psi_scaled_inequality} that
\[ 2 \psi_{\theta_0}(\lambda v) + 2 \leq \psi_{\theta_0}(v). \]
Therefore
\begin{align*} 
\abs{\bar{f}(t,x,v)} &\leq \paren{1 + \frac{\lambda}{2}} \bracket{1 + \lambda \psi_{\theta_0}(\lambda v) + \lambda/2} 
\\ &\leq \paren{1 + \frac{\lambda}{2}} \bracket{1 + \frac{\lambda}{2} \psi_{\theta_0}(v) - \lambda + \lambda/2} 
\\ &\leq 1 + \lambda \psi_{\theta_0}(v).
\end{align*}

This completes the proof.  

\end{proof}

Theorem~\ref{thm:main} is proven by iteratively applying this Lemma~\ref{thm:oscillation} to an appropriately scaled function.  


\appendix
\section{Some Technical Lemmas}\label{sec:appendix}

We prove here the averaging lemma used throughout this paper.  
This lemma is an immediate corollary of \cite{bezard} Theorem 6.  It is merely a localization of that result.  
\begin{lemma}[Averaging Lemma]\label{thm:avg_lemma}
Let $\Omega$ be an open subset of space-time $\R \times \R^n$, and $\bar{\Omega}$ a compact subset of $\Omega$.  

For any smooth function $\eta \in \Ctest(\R^n)$ and any $m \in \R^+$, there exists a constant $C = C(n,m,\eta, \bar{\Omega}, \Omega)$ and a constant
\[ \alpha = \frac{1}{2(1+m)} \]
such that the following is true:

For any two functions $f$ and $g$ in $L^2(\Omega \times \R^n)$ satisfying
\[ \kinet f = g, \]
it is true that
\[ \norm{\int \eta f \,dv}_{H^\alpha(\bar{\Omega})} \leq C \paren{\norm{f}_{L^2(\Omega \times \R^n)} + \norm{\bessel^{-m/2} g}_{L^2(\Omega\times\R^n)}}. \]
%
\end{lemma}

By $\norm{g}_{H^\alpha(\bar{\Omega})}$, we mean the infimum of $\norm{\tilde{g}}_{H^\alpha(\R^{n+1})}$ over all extensions $\tilde{g}$ of $g$ to $\R^{n+1}$.  

\begin{proof}
Let $\phi(t,x)$ be a smooth function supported on $\Omega$ and identically equal to 1 on $\bar{\Omega}$.  Then
\[ \kinet (\phi f) = \phi g + f \kinet \phi. \]
By \cite{bezard} Theorem 6,
\[ \norm{\phi \int \eta f \,dv}_{H^\alpha(\R \times \R^n)} \leq C \paren{\norm{\phi f}_{L^2(\R \times \R^n \times \R^n)} + \norm{\bessel^{-m/2} \paren{\phi g + f \kinet \phi}}_{L^2(\R \times \R^n \times \R^n)}}. \]
Because $\bessel^{-m/2}$ is a bounded operator from $L^2$ to $L^2$, and because $\phi$ is a smooth function supported on $\Omega$ and depending only on $t$ and $x$, 
\[ \norm{\bessel^{\frac{-m}{2}} \paren{\phi g + f \kinet \phi}}_{L^2(\R^{1+n+n})} \leq C(\phi) \norm{\bessel^{\frac{-m}{2}}g}_{L^2(\Omega\times \R^n)} + C(m,\phi)\! \norm{f}_{L^2(\Omega\times\R^n)}\!. \]

The result follows.  
\end{proof}

The following is a technical lemma about the geometry of cones.  We use it at the very end of the proof of Proposition~\ref{thm:DG2}.  

\begin{figure}[h]
\includegraphics[width=.5 \textwidth]{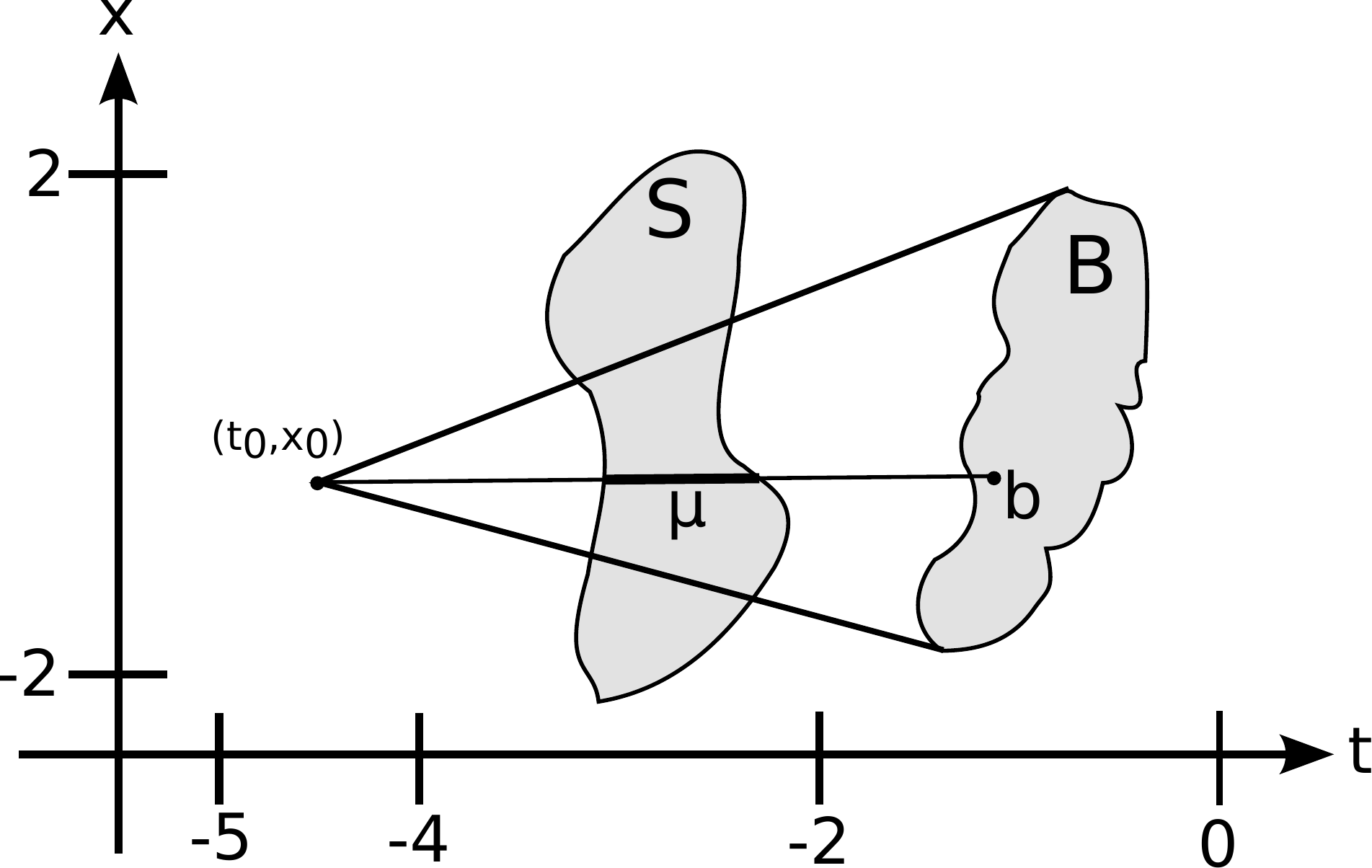}
\caption{A diagram showing the assumptions of Lemma~\ref{thm:cone}.}
\end{figure}

\begin{lemma}\label{thm:cone}
Let $\mathcal{C} \subseteq \R \times \R^n$ be a cone from a vertex $(t_0,x_0) \in [-5, -4] \times B_2$ to a base set $B \subseteq [-2,0]\times B_2$.  
Let $S$ be a subset of $\R\times\R^n$ such that for each $b \in B$, the line segment connecting $(t_0,x_0)$ to $b$ intersects $S$ on a set with Hausdorff $\mathcal{H}^1$ measure at least $\mu$.  

Then 
\[ \abs{\mathcal{C} \cap S} \geq \frac{ |B| \mu^2}{80}. \]
\end{lemma}

\begin{proof}
Let $A(t)$ be the cross-sectional area of our cone at time slice $t$.  If $\mathcal{H}^n$ is the Hausdorff measure of dimension $n$, we write
\[ A(t) = \mathcal{H}^n\paren{\mathcal{C} \cap \bracket{\{t\}\times \R^n}}. \]
By the nature of cones, $A(t_0)=0$, $A$ is affine for $t_0 < t < -2$, then sub-affine for $-2<t<0$, and $A(t)=0$ for $t > 0$.  Specifically,
\[ A(t) = \frac{A(-2)}{-2-t_0} \paren{t-t_0} \qquad t_0 < t < -2, \]
\[ A(t) \leq \frac{A(-2)}{-2-t_0} \paren{t-t_0} \qquad -2 \leq t. \]

Since $B$ is contained in $\mathcal{C} \cap [-2,0]\times \R^n$, 
\[ |B| \leq \int_{-2}^0 A(t) \,dt \leq \int_{-2}^0 \frac{A(-2)}{-2-t_0}\paren{t-t_0} \,dt = \frac{A(-2)}{-2-t_0} \bracket{t_0^2 - (2+t_0)^2}/2 \leq 4 A(-2). \]

This means that 
\[ A(-2) \geq \frac{|B|}{4}. \]

Now we have a lower bound on the size of the cone, so for $t_0 \leq t \leq -2$
\begin{equation}\label{lower_bound_on_A} 
A(t) \geq \frac{|B|}{4(-2-t_0)}(t - t_0). 
\end{equation}

%

Consider the map from $B$ to $\{0\}\times\R^n$ given by stereographic projection from the point $(t_0,x_0)$, and let $db$ be a probability measure on $B$ proportional to the pullback of $\mathcal{H}^n \rest_{\{0\}\times\R^n}$ under this projection.  
Then $db$ represents the proportion of any time-slice of $\mathcal{C}$ generated by rays through a given portion of $B$.  

To find the measure of $\mathcal{C} \cap S$, we must ask how much each time slice intersects $S$, or in integral form
\[ \abs{\mathcal{C} \cap S} = \int_{t_0}^{0} A(t) \int_{b \in B} \indic{(t,x) \in \mathcal{C} \cap S} \,dbdt. \]
By Fubini, this becomes
\begin{equation}\label{measure_of_intersection} 
\abs{\mathcal{C} \cap S} = \int_{b \in B} \int_{t_0}^0 A(t) \indic{(t,x) \in \mathcal{C} \cap S} \,dtdb. 
\end{equation}

From the definition of $\mu$ and the arc length formula,
\[ \mu \leq \int_{t_0}^0 \indic{(t,x) \in \mathcal{C} \cap S} \sqrt{1 + |b-x_0|^2/(-2-t_0)^2} \,dt \leq \sqrt{5} \int_{t_0}^0 \indic{(t,x) \in \mathcal{C} \cap S}. \]

Because $A(t)$ is increasing and $\indic{(t,x) \in \mathcal{C} \cap S}$ integrates to at least $\mu/\sqrt{5}$, 
\[ \int_{t_0}^0 A(t) \indic{(t,x) \in \mathcal{C} \cap S} \,dt \geq \int_{t_0}^{t+\mu/\sqrt{5}} A(t) \,dt. \]
From this bound, \eqref{measure_of_intersection}, and \eqref{lower_bound_on_A} we can at last compute
\[ \abs{\mathcal{C} \cap S} \geq \frac{|B|}{4(-2-t_0)} \int_{t_0}^{t_0 + \mu/\sqrt{5}} (t-t_0) \,dt = \frac{|B|}{4(-2-t_0)} \frac{\mu^2}{10} \geq \frac{|B|\mu^2}{80}. \]
\end{proof}

The following lemma is a commonly known fact about mollifiers.  Despite being known, a proof is surprisingly difficult to find in the existing literature.  Therefore, in the interest of completeness, we prove it here.   

\begin{lemma}\label{thm:convolution_estimate}
Let $\eta\in \Ctest(\R^n)$ be such that the sequence $\eta_\eps(v) = \eps^{-n} \eta(v/\eps)$ is an approximation to the identity.  There exists a constant $C=C(n,s,\eta)$ such that, for any $g \in H^s(\R^n)$, 
\[ \norm{g - g\ast \eta_\eps}_{L^2(\R^n)} \leq C \norm{g}_{H^s(\R^n)} \eps^s. \]
\end{lemma}

\begin{proof}
The bound is easy to compute by taking the Fourier transform and using Plancharel's theorem:
\begin{align*} 
\norm{g - g\ast \eta_\eps}_{L^2}^2 &= \int \hat{g}^2 \paren{1 - \hat{\eta_\eps}}^2 \,d\xi 
\\ &\leq \int (1+\xi^2)^s \hat{g}^2 \,d\xi \,\, \sup_{\xi} \frac{\abs{1-\hat{\eta_\eps}(\xi)}^2}{(1+\xi^2)^s}
\\ &= \norm{g}_{H^s(\R^n)}^2 \,\, \sup_{\xi} \frac{\abs{1-\hat{\eta_\eps}(\xi)}^2}{(1+\xi^2)^s}.
\end{align*}

Since $\eta \in \Ctest$, the fourier transform $\hat{\eta}$ is Lipschitz with some constant $\bar{C}$. Thus $\hat{\eta_\eps}(\xi) = \hat{\eta}(\eps \xi)$ is Lipschitz with constant $\bar{C}\eps$.  Since $\eta_\eps$ is an approximation to the identity, $\hat{\eta_\eps}(0) = 1$ and $\abs{\hat{\eta_\eps}(\xi)} \leq 1$ for all $\xi$.  Thus
\[ \abs{1 - \hat{\eta_\eps}(\xi)} \leq \min(2,\bar{C}\eps|\xi|). \]
%

The function $\frac{\min(2,\bar{C}\eps\xi)^2}{(1+\xi^2)^s}$ achieves its maxumum value at the critical point $\bar{C}\eps|\xi| = 2$, and that maximum value is 
\[ \frac{2^2}{\paren{1+\paren{\frac{2}{\bar{C}\eps}}^2}^s} = \frac{4 \eps^{2s}}{\paren{\eps^2 + 4/\bar{C}^2}^s} \leq C \eps^{2s}. \]
\end{proof}

\bibliographystyle{plain}
\bibliography{NFP-references.bib}

\end{document}